\documentclass[a4paper,oneside,11pt] {amsart}
\reversemarginpar  

\usepackage{amsthm}
\usepackage{amsfonts}
\usepackage{amsmath}
\usepackage{amssymb}
\usepackage{mathrsfs}
\usepackage{epsfig}
\usepackage{pst-plot}
\usepackage{ifthen}
\usepackage[active]{srcltx}

\newcommand{\ncom}{\newcommand}
\ncom{\ul}{\underline}
\ncom{\ol}{\overline}
\ncom{\bq}{\begin{equation}}
\ncom{\eq}{\end{equation}}
\ncom{\beqn}{\begin{eqnarray*}}
\ncom{\eeqn}{\end{eqnarray*}}
\ncom{\beq}{\begin{eqnarray}}
\ncom{\eeq}{\end{eqnarray}}
\ncom{\nno}{\nonumber}
\ncom{\rar}{\rightarrow}
\ncom{\Rar}{\Rightarrow}
\ncom{\noin}{\noindent}
\ncom{\bc}{\begin{centre}}
\ncom{\ec}{\end{centre}}
\ncom{\sz}{\scriptsize}
\ncom{\rf}{\ref}
\ncom{\sgm}{\sigma}
\ncom{\Sgm}{\Sigma}
\ncom{\dt}{\delta}
\ncom{\Dt}{\Delta}
\ncom{\la}{\lambda}
\ncom{\La}{\Lambda}

\ncom{\al}{\alpha}
\ncom{\be}{\beta}
\ncom{\de}{\delta}
\ncom{\ga}{\gamma}
\ncom{\lmd}{\lambda}
\ncom{\Lmd}{\Lambda}
\ncom{\eps}{\epsilon}

\def \N{\mathbb{N}}

\ncom{\pcc}{\stackrel{P}{>}}
\ncom{\dist}{{\rm\,dist}}
\ncom{\sspan}{{\rm\,span}}
\ncom{\re}{{\rm Re\,}}
\ncom{\im}{{\rm Im\,}}
\ncom{\sgn}{{\rm sgn\,}}
\ncom{\ba}{\begin{array}}
\ncom{\ea}{\end{array}}
\ncom{\eop}{\hfill{{\rule{2.5mm}{2.5mm}}}}
\ncom{\eoe}{\hfill{{\rule{1.5mm}{1.5mm}}}}
\ncom{\eof}{\hfill{{\rule{1.5mm}{1.5mm}}}}
\ncom{\hone}{\mbox{\hspace{1em}}}
\ncom{\htwo}{\mbox{\hspace{2em}}}
\ncom{\hthree}{\mbox{\hspace{3em}}}
\ncom{\hfour}{\mbox{\hspace{4em}}}
\ncom{\hsev}{\mbox{\hspace{7em}}}
\ncom{\vone}{\vskip 2ex}
\ncom{\cH}{{\mathcal H}}
\ncom{\vtwo}{\vskip 4ex}
\ncom{\vonee}{\vskip 1.5ex}
\ncom{\vthree}{\vskip 6ex}
\ncom{\vfour}{\vspace*{8ex}}
\ncom{\norm}{\|\;\;\|}
\ncom{\integ}[4]{\int_{#1}^{#2}\,{#3}\,d{#4}}
\ncom{\inp}[2]{\langle{#1},\,{#2} \rangle}
\ncom{\Inp}[2]{\big\langle{#1},\,{#2} \big\rangle}
\ncom{\vspan}[1]{{{\rm\,span}\#1 \}}}
\ncom{\dm}[1]{\displaystyle {#1}}
\ncom{\Hom}{\operatorname{Hom}}
\ncom{\Hol}{\operatorname{Hol}}

\newcommand{\summ}[2]
{\sum\limits_{\mbox{\tiny $\begin{array}{c}
                            {#1} \\
                             {#2}
                          \end{array}$}}}

\ncom{\defin} {\overset {\text {\rm def} }{=}}

\newtheorem{theorem}{\bf Theorem}[section]
\newtheorem{example}[theorem]{\bf Example}
\newtheorem{proposition}[theorem]{\bf Proposition}
\newtheorem{corollary}[theorem]{\bf Corollary}
\newtheorem{lemma}[theorem]{\bf Lemma}

\newtheorem{question}[theorem]{\bf Question}

\newtheoremstyle
    {remarkstyle}
    {}
    {11pt}
    {}
    {}
    {\bfseries}
    {:}
    {     }
    {\thmname{#1} \thmnumber{#2} }

\theoremstyle{remarkstyle}

\newtheorem{remark}[theorem]{\bf Remark}
\newtheorem{definition}[theorem]{\bf Definition}

\ncom{\bbeta}{{\tilde\beta}}
\ncom{\rr}{\tau}

\renewcommand{\epsilon}{\varepsilon}
\renewcommand{\kappa}{\varkappa}

\begin{document}

\baselineskip=16pt

\title
{Spherical Tuples of Hilbert Space Operators}
\author[S. Chavan]{Sameer Chavan}
\address{Department of Mathematics and Statistics\\
Indian Institute of Technology Kanpur, India}
\email{chavan@iitk.ac.in}
\author[D. Yakubovich]{Dmitry Yakubovich}
\address{Department of Mathematics\\
Universidad Aut${\mbox{\'{o}}}$noma de Madrid, Spain}
\email{dmitry.yakubovich@uam.es}
\date{}
\begin{abstract}


We introduce and study a class of operator tuples
in complex Hilbert spaces, which we call spherical tuples.
In particular, we
characterize spherical multi-shifts, and more generally,
multiplication tuples on RKHS. We further use these
characterizations to describe various spectral parts including the
Taylor spectrum. We also find a criterion for
the Schatten $S_p$-class membership of cross-commutators of spherical
$m$-shifts. We show, in particular, that
cross-commutators of non-compact spherical $m$-shifts cannot belong
to $S_p$ for $p \le m$.

We specialize our results to some well-studied classes of
multi-shifts. We prove that
the cross-commutators of a spherical joint  $m$-shift,
which is a $q$-isometry or a $2$-expansion,
belongs to $S_p$ if and only if $p > m.$ We further
give an example of a spherical jointly hyponormal $2$-shift, for which
the cross-commutators are compact but not in $S_p$ for any $p <\infty$.
\end{abstract}

\renewcommand{\thefootnote}{}



\maketitle

\setcounter{tocdepth}{2} \tableofcontents


\section{Introduction}

The motivation for the present paper comes from different
directions.
Firstly, as there is considerable literature on circular
operators (refer to \cite{Sh}, \cite{Ge}, \cite{A-H-H-K}, \cite{Ml},
\cite{O}), it is natural to look for the higher-dimensional analogs
of circular operators.
There are of course two possible analogs,
namely, {\it poly-circular tuples} and {\it spherical tuples}.
%
%
Multi-variable weighted shifts (for short, {multi-shifts}) form a
subclass of the class of poly-circular tuples, and indeed, there
are some important papers on this latter class (see, for instance,
\cite{J-L}, \cite{Cu-Sa}, \cite{G-H-X}, \cite{D-S}).
There are also several important papers on multivariable
weighted shifts that are spherical, see \cite{Bo}, \cite{At-2},
\cite{Ar}, \cite{B-M}, \cite{A-Z}, \cite{G-H-X}, \cite{G-R-S}.
However, the higher-dimensional
counter-parts of many important results in the masterful
exposition \cite{Sh} by A. Shields are either unknown or not
formulated.
The main objective of this paper is to introduce
spherical operator tuples in an abstract way and
to study some of their basic properties, as well as properties of
spherical multi-variable weighted shifts, which form a subclass
of this class.
%
%
%
%

One of our motivations is the following  phenomenon
concerning multi-dimensional cross-comm\-ut\-at\-ors and
Hankel operators, which is often referred to as ``cut-off'':
in several particular situations, these operators cannot be too small
unless they are zero, see
\cite{JansWo}, \cite{Wallst}, \cite{Z-Sp}, \cite{D-V},
\cite{Li93}, \cite{Y}. More recently,
related questions have been studied in relation with the multi-variable
Berger-Shaw theory and the so-called Arveson conjecture,
see, for instance, \cite{D-Y}, \cite{Ar-p}, \cite{Eshm}, \cite{F-X},
\cite{G-KW-Zh}, \cite{KSh} and others.
In our context, we prove that
cross-commutators of non-compact spherical $m$-shifts do not belong
to $S_p$ for $p \le m$.


If $\mathbb N$ stands for the set of non-negative integers, we
denote by ${\mathbb N}^m$ for the cartesian product ${\mathbb N}
\times \cdots \times {\mathbb N}~(m ~\text{times}).$ Let $p = (p_1,
\cdots, p_m)$ and $n = (n_1, \cdots, n_m)$ be in ${\mathbb N}^m.$ We
write $p \leq n$ if $p_j \leq n_j$ for $j=1, \cdots, m.$ For $p\leq
n,$ $n \choose p$ is understood to be the product ${n_1 \choose p_1}
\cdots {n_m \choose p_m}$ and $|p|$ is understood to be $p_1 +
\cdots + p_m.$

For a Hilbert space
$\mathcal H$, let $\mathcal H^{(m)}$ denote the orthogonal direct
sum of $m$ copies of $\mathcal H.$
%
%
Let ${B}({\mathcal H})$ denote the Banach algebra of
bounded linear operators on $\mathcal H.$
If the opposite is not specified, all the operators we consider will be
assumed linear and bounded.

 If $T = (T_1, \cdots, T_m)$
is an $m$-tuple of commuting bounded linear operators $T_j~(1 \leq j
\leq m)$ on $\mathcal H$ then we set $T^*$ to be $(T^*_1,
\cdots, T^*_m)$ and $T^p$ to be $T^{p_1}_1\cdots T^{p_m}_m.$

The main object of our interest in this paper is the class of so-called spherical tuples.

\begin{definition} Let $T$ be an $m$-tuple of commuting bounded linear
operators $T_1, \cdots, T_m$ on an infinite dimensional Hilbert
space $\mathcal H.$ Let $\mathcal U(m)$ denote the group of
complex $m \times m$ unitary matrices. For $U=(u_{jk})_{1 \leq j,
k \leq m} \in \mathcal U(m)$, the commuting operator $m$-tuple
$T_U$ is given by \beq (T_U)_{j}= \sum_{k=1}^mu_{jk}T_k~(1 \leq j
\leq m). \eeq We say that $T$ is {\it spherical} if for every $U
\in \mathcal U(m),$ there exists a unitary operator $\Gamma(U)\in
B(\mathcal H)$ such that $\Gamma(U)T_j=(T_U)_j\Gamma(U)$ for all
$j=1, \cdots, m.$ If, further, $\Gamma$ can be chosen to be a
strongly continuous unitary representation of $\mathcal U(m)$ on
$\mathcal H$ then we say that $T$ is {\it strongly spherical}.
\end{definition}

\begin{remark}
\label{rk1.2}
Let $T=(T_1, \cdots, T_m)$ be a spherical $m$-tuple.
\begin{enumerate}

\item Any permutation
of $T$ is unitarily equivalent to $T.$ In particular, $T_j$ is
unitarily equivalent to $T_k$ for any $1 \leq j, k \leq m.$

%

\item For any unital $*$-representation
$\pi: B({\mathcal H})\to B({\mathcal H_1})$,
$\pi(T):=(\pi(T_1), \cdots, \pi(T_m))$ is also a
spherical $m$-tuple. Indeed, since $\pi$ sends unitaries to unitaries,
$\pi(\Gamma(U))\pi(T_j)=(\pi(T)_U)_j\pi(\Gamma(U))$ for all $j=1,
\cdots, m$, and $\pi(\Gamma(U))$ is unitary.
We also observe that $T^*=(T^*_1, \cdots, T^*_m)$ is spherical.

\end{enumerate}
\end{remark}

Let $T=(T_1, \cdots, T_m)$ be an $m$-tuple
of commuting bounded linear operators
on a Hilbert space $\mathcal H$.
Let $D_T$ denote the linear transformation from
$\mathcal H$ to $\mathcal H^{(m)}$, given by
\beqn
D_{T}h := (T_1h, \cdots, T_mh)~(h \in \mathcal H). 
\eeqn
Note that
$\mbox{ker}(D_{T})=\bigcap_{i=1}^{m}\text{ker}(T_i).$

Next, we need to invoke the basics of the theory of multi-shifts
\cite{J-L}. First a definition.

Let $T$ be an $m$-tuple of commuting operators $T_1, \cdots, T_m$ on
a Hilbert space $\mathcal H.$ A closed subspace $\mathcal M$ of
$\mathcal H$ is said to be {\it cyclic} for $T$ if \beqn \mathcal H
= \bigvee \{T^nx : x \in \mathcal M,~ n \in \mathbb N^m\}. \eeqn We
say that $T$ is {\it cyclic} with {\it cyclic vector} $x$ if the
subspace spanned by $x$ is cyclic for $T.$

Let $\left\{{w^{(j)}_n} : 1 \leq j \leq m, n \in {\mathbb
N}^m\right\}$ be a multi-sequence of complex numbers. An {\it $m$-variable
weighted shift} $T = (T_1, \cdots, T_m)$ with respect to an
orthonormal basis $\{e_n\}_{n \in \mathbb N^m}$ of a Hilbert space
$\mathcal H$ is defined by \beqn T_je_n \mathrel{\mathop:}=
w^{(j)}_n e_{n + \epsilon_j}~(1 \leq j \leq m), \eeqn where
$\epsilon_j$ is the $m$-tuple with $1$ in the $j$th place and zeros
elsewhere.
The notation $T : \{{w^{(j)}_n}\}_{ n \in \mathbb N^m}$ will mean
%
%
that $T$ is the $m$-variable weighted shift tuple with weight
multi-sequence $\left\{{w^{(j)}_n} : 1 \leq j \leq m, n \in {\mathbb
N}^m\right\}$. Notice that $T_j$ commutes with $T_k$ if and only if
$w^{(j)}_nw^{(k)}_{n + \epsilon_j}=w^{(k)}_nw^{(j)}_{n +
\epsilon_k}$ for all $n \in {\mathbb N}^m.$ By \cite[Corollary
9]{J-L}, $T$ is bounded if and only if
$$
\sup \big\{|w^{(j)}_n| : 1 \leq j
  \leq m, n \in {\mathbb N}^m \big\} < \infty.
$$
We always assume that the weight multi-sequence of $T$ consists of
positive numbers and that $T$ is commuting. Note that $T :
\{{w^{(j)}_n}\}_{ n \in \mathbb N^m}$ is cyclic with cyclic vector
$e_0.$


Let $T : \{{w^{(j)}_n}\}_{ n \in \mathbb N^m}$ be an $m$-variable
weighted shift. Define $\beta_n = \|T^ne_0\|~(n \in \mathbb N^m)$ and
consider the Hilbert space $H^2(\beta)$ of formal power series
\beqn
f(z)=\sum_{n \in \mathbb N^m}a_n z^n
\eeqn
such that
\beqn
\|f\|^2_{\beta}=\sum_{n \in \mathbb N^m}|a_n|^2 \beta^2_n
< \infty.
\eeqn
It follows from \cite[Proposition 8]{J-L} that any
$m$-variable weighted shift $T$  is unitarily equivalent to the
$m$-tuple $M_z=(M_{z_1}, \cdots, M_{z_m})$ of multiplication by the
co-ordinate functions $z_1, \cdots, z_m$ on the corresponding space
$H^2(\beta)$.
Notice that the linear set of polynomials in $z_1, \cdots, z_m$ (that
is, formal power series with finitely many non-zero coefficients) is
dense in $H^2(\beta)$. Equivalently, $M_z$ is cyclic with cyclic
vector the constant formal power series $1$ (that is, the formal
series $\sum_{n \in \mathbb N^m} a_n z^n$, for which $a_n=0$ for all
non-zero $n \in \mathbb N^m$ and $a_{0}=1$). The relation between
weights $w_n^{(j)}$ and the sequence $ \beta_n$ is given by
\bq
\label{w-beta}
w_n^{(j)}= {\beta_{n+\eps_j}}/{\beta_n},
\qquad 1\le j\le m, \; n\in \mathbb{N}^m.
\eq
Note further that
$\mbox{ker}(D_{M^*_z})$ is spanned by the constant formal power series $1.$

Recall that all formal power series in $H^2(\beta)$ converge
absolutely on the point-spectrum $\sigma_p(M^*_z)$ of the adjoint
$m$-tuple $M^*_z$ of $M_z$ \cite[Propositions 19 and 20]{J-L}. In
particular, $H^2(\beta)$ may be realized as a {\it reproducing
kernel Hilbert space} (RKHS) with {\it reproducing kernel} $\kappa :
\sigma_p(T^*) \times \sigma_p(T^*) \rar \mathbb C$ given by
\beq
\label{rkernel} \kappa(z, w) = \sum_{n \in \mathbb
N^m}\overline{w}^n {z^n}/{\beta^2_n}
\quad (z, w \in \sigma_p(T^*)),
\eeq
assuming that $\sigma_p(M^*_z)$ has non-empty
interior.
Conversely,
as it follows from Theorem~\ref{th:toral} below,
the multiplication $m$-tuple $M_z$ acting in a RKHS
$\mathscr H$ with reproducing kernel $\kappa$ of the form
\eqref{rkernel} is unitarily equivalent to an $m$-variable weighted
shift on $H^2(\beta)$ if all complex polynomials in $z_1,
\cdots, z_m$ are contained in $\mathscr H$.
Notice that the norm in $H^2(\beta)$ has poly-circular symmetry:
$\|f(\zeta \cdot z)\|_\beta=\|f(z)\|_\beta$ for any $f\in H^2(\beta)$
and any $\zeta\in \mathbb T^m$, where
$\zeta \cdot z = (\zeta_1 z_1, \dots,  \zeta_m z_m)$.
So if the largest open set where all series in $H^2(\beta)$
converge is not empty, it is a Reinhardt domain.


We denote by $\mathbb B_R$ the open ball centered at the origin and of
radius $R > 0$:
$$
\mathbb B_R :=
\{ z= (z_1, \cdots, z_m)
\in \mathbb C^m : \quad \|z\|^2_2=|z_1|^2 + \cdots + |z_m|^2 < R^2\}.$$ The sphere
centered at the origin and of radius $R > 0$ is denoted by $\partial
\mathbb B_R.$ For simplicity, the unit ball $\mathbb B_1$ and the
unit sphere $\partial \mathbb B_1$ are denoted respectively by
$\mathbb B$ and $\partial \mathbb B.$

Let us discuss three basic examples of (spherical) weighted $m$-variable shifts, with which we are primarily concerned.
\begin{example} \label{ex1.3}
{\rm
For any real number $p>0$, let $\mathscr H_{p}$ be the RKHS of
holomorphic functions on the unit ball $\mathbb B$ with reproducing
kernel
\beqn
\kappa_p(z, w)=\frac{1}{(1- \inp{z}{w})^p}~(z, w \in
\mathbb B).
\eeqn
If $M_{z, p}$ denotes the multiplication tuple on $\mathscr H_p$ then it is unitarily
equivalent to the weighted shift $m$-tuple with weight sequence
\bq
\label{Sz-Be-Dru}
w^{(i)}_{n, p}=\sqrt{\frac{n_i + 1}{|n|+p}}~(n \in \mathbb
N^m, i=1, \cdots, m).
\eq
The RKHS's $\mathscr H_{m}, \mathscr H_{m+1},
\mathscr H_{1}$ are, respectively, the {\it Hardy space} $H^2(\partial \mathbb B)$,
the {\it Bergman space} $A^2(\mathbb B)$, the {\it Drury-Arveson
space} $H^2_m$.
The multiplication tuples $M_{z, m},
M_{z, m+1}, M_{z, 1}$
are commonly known as the {\it
Szeg\"o $m$-shift}, the {\it Bergman $m$-shift}, the
{\it Drury-Arveson $m$-shift} respectively.
The spaces $\mathscr H_p$ have been studied in many papers.
In \cite{VolbWick}, a characterization of
Carleson measures in these spaces has been given.
In \cite{Kapt}, the spaces
$\mathscr F_q=\mathscr H_{1+m+q}$ have been studied;
in particular, a kind of model theorem and von Neumann
inequalities related to these spaces for row contractions
is established there and some K-theory results are proved for
the corresponding Toeplitz algebras. In this work, a scale of
Dirichlet-type spaces corresponding to $q<0$ is also considered, but their definition is
different, and they do not belong to the collection of spaces $\mathscr H_p$.

As it is proved in \cite{At-2}, $M_{z, p}$ is subnormal for any $p\ge m$.
In fact, $M_{z, p}$ is jointly subnormal if and only if $p \ge m$, see the discussion after Theorem \ref{le5.4}.
}
\end{example}

\smallskip

The paper is organized as follows.
In the second section, we present various characterizations of
spherical tuples. The main results of this section are Theorem \ref{th:w-sh}, where
we characterize $m$-variable weighted shifts (equivalently,
multiplication $m$-tuples), which are
spherical, and Theorem \ref{th:spher-tpl}, which gives abstract conditions,
when an arbitrary spherical operator $m$-tuple is unitarily equivalent to
a multiplication $m$-tuple.
We also discuss some examples.
%
%
In Section~3, we describe various spectral parts of spherical
multi-shifts, including the Taylor spectrum. In particular, we obtain
refinements of some results in \cite{G-H-X}. In Section 4, we
provide a sufficient and necessary condition for the Schatten
$p$-class membership of cross-commutators of spherical $m$-shifts.
We deduce that for a noncompact $m$-tuple $M_z$,
if $[M_{z_j}^*, M_{z_k}]\in S_p$ for all $j,k$, then
$p>m$ (which is a manifestation of the cut-off).
Here $[A, B]$ stands for the commutator $AB-BA$ of
operators $A, B$ on a space $\mathcal H$.
The results of Sections~3 and~4 rely heavily on the results of
Section~2. In the last Section~5, we mainly discuss the cut-off
phenomenon for some special classes of spherical multi-shifts, such
as $q$-expansions, $q$-isometries and jointly hyponormal tuples.

\section{Spherical Tuples}

Let $\mathbb C[z]$ stand for the vector space of analytic
polynomials in $z_1, \cdots, z_m$. We define
\beqn
\Hom(k)=\Big\{p \in \mathbb{C}[z]: p(z)=\sum_{|n|=k}a_kz^n \Big\}.
\eeqn
For a polynomial $p\in \mathbb{C}[z]$ and an integer $k\ge
0$, we denote by $p_{[k]}\in \Hom(k)$ the homogeneous part of
$p$ of degree $k$. More generally, $f_{[k]}$ stands for the
homogeneous part $\sum_{|n|=k}a_kz^n$ of a formal power series
$f(z)=\sum_{n \in \mathbb N^m}a_n z^n.$

Let $\sigma$ denote the normalized surface area measure on the unit sphere
$\partial \mathbb B.$
We often use the short notation $L^2(\partial \mathbb B)$ for the
Hilbert space $L^2(\partial \mathbb B, \sigma)$ of
$\sigma$-square-integrable ``functions" on $\partial \mathbb B.$

The first theorem of this section provides a handy characterization
of spherical multi-shifts. The multi-shifts with weight
multi-sequence given by \eqref{w-i-n} arise naturally in the study
of reproducing $\mathbb C[z_1, \cdots, z_m]$-modules with $\mathcal
U(m)$-invariant kernels, refer to \cite[Section 4]{G-H-X}.
\begin{theorem}
\label{th:w-sh} Let $M_z$ be a bounded multiplication $m$-tuple in
$H^2(\beta).$ Then $M_z$ is spherical if and only if the norm
$\|\cdot\|_{\beta}$ on $H^2(\beta)$ can be expressed as \bq
\label{H2-beta-L2B}
\|f\|^2_{\beta}=\sum_{k=0}^{\infty}{\bbeta}^2_k
\|f_{[k]}\|^2_{L^2(\partial \mathbb B)}~(f \in H^2(\beta)) \eq for
a sequence ${\bbeta_0}, {\bbeta_1}, {\bbeta_2}, \cdots,$
of positive numbers. If this happens then $M_z$ is unitarily
equivalent to the $m$-variable weighted shift $T :
\{{w^{(i)}_n}\}_{ n \in \mathbb N^m}$ with the weight sequence \bq
\label{w-i-n}
w^{(i)}_n=\frac{{\bbeta}_{|n|+1}}{{\bbeta}_{|n|}}
\sqrt{\frac{n_i+1}{|n|+m}}~( n \in {\mathbb N}^m,  1 \leq i \leq m).
\eq
In this case, the sequence $\beta_n = \|z^n\|_{\beta}$ can be expressed as
\bq \label{beta-n}
\beta_n = {\bbeta}_{|n|}
\sqrt{\frac{(m-1)!\, n!}{(m-1+|n|)!}}~ (n \in \mathbb N^m).
\eq

\end{theorem}
\begin{remark}
\label{rk2.2} Whenever
$\{\beta_n\}_{n\in{\mathbb N}^m}$ is a
multi-sequence, which gives rise to a spherical tuple $M_z$,
we will denote by $\{\bbeta_k\}_{k\in\mathbb N}$
the corresponding scalar weight sequence, related to $\beta$ via
formula \eqref{beta-n}.
\end{remark}

\begin{definition}
\label{df4.1} Let $T : \{w^{(i)}_n\}_{ n \in \mathbb N^m}$ be a
spherical $m$-variable weighted shift and let
$\{\bbeta_k\}_{k\in\mathbb N}$ be the corresponding scalar weight
sequence. Then the {\it shift associated with $T$} is the
one-variable weighted shift $T_{\delta} : \{\delta_{k}\}_{k \in
\mathbb N},$ where
\[
\delta_k :=\frac {{\bbeta}_{k+1}} {{\bbeta}_{k}}, \quad k \in \mathbb N.
\]
\end{definition}

It is easy to see that the following statements are equivalent:
\begin{enumerate}
\item A scalar weight sequence
$\{\bbeta_k\}$ gives rise to a \textit{bounded}
spherical $m$-tuple $M_z$ on $H^2(\beta),$ where $\beta$ is given by \eqref{beta-n};
\item The spherical $m$-variable shift $T : \{w^{(i)}_n\}_{ n \in \mathbb N^m}$
is bounded;
\item $\sup_{k\ge 0} \delta_k <\infty$;
\item The one-variable shift $T_\delta$, associated with $T$, is bounded.
%
%
\end{enumerate}
When dealing with a spherical multiplication $m$-tuple $M_z$ and with
the corresponding $m$-variable weighted shift $T,$ we will always assume
that the condition (3) above holds.

For an $m$-tuple $T$ of commuting bounded linear operators $T_1, \cdots, T_m$ on $\mathcal H,$ let
\[
Q_T(I) := \sum_{j=1}^mT^*_jT_j.
\]
\begin{remark}
\label{rk2.3}
Let $M_z$ be a bounded multiplication $m$-tuple in
$H^2(\beta).$
Then $Q_{M_z}(I)=I$ if and only if
$M_z$ is the Szeg\"o $m$-shift.
Further, the defect operator $I-Q_{M^*_z}(I)$ is an orthogonal
projection if and only if $M_z$ is the Drury-Arveson $m$-shift.
\end{remark}

%
%
%

The next result characterizes all multi-shifts within the whole
class of spherical tuples and should be combined with the above
Theorem \ref{th:w-sh}.
Recall that for an $m$-tuple $S=(S_1, \cdots, S_m),$
 $\mbox{ker}(D_{S^*})=\bigcap_{i=1}^{m}\text{ker}(S^*_i).$
\begin{theorem}
\label{th:spher-tpl} Let $T$ be a commuting, bounded spherical
operator $m$-tuple on
a Hilbert space $\mathcal H$. Then the following assertions
are equivalent.

\textbf{(1)} $\ker(D_{T^*})$ is a one-dimensional cyclic subspace
for $T$;

\textbf{(2)} $T$ is unitarily equivalent to an $m$-variable weighted
shift;

\textbf{(3)} $T$ is unitarily equivalent to a multiplication
$m$-tuple $M_z$ on a space $H^2(\beta)$.
\end{theorem}

Before we turn to the proofs of Theorems \ref{th:w-sh} and
\ref{th:spher-tpl}, let us see a couple of instructive examples.
\begin{example}
{\rm Let $M_z$ be a bounded spherical multiplication $m$-tuple on a
space $H^2(\beta)$ and suppose that the ball $B_R$, where all power
series in  $H^2(\beta)$ converge has positive radius (see Theorem
\ref{th:3.1}(2) below for the description of $R=r(M_z)$ in terms of
$\beta_k$'s). Fix an integer $s>0$. Then the set
$$
H^2(\beta)_s=\big\{f\in H^2(\beta): D^\alpha f(0)=0\quad \text{for
all}\; \alpha, \;|\alpha|<s\big\}
$$
is a closed subspace of $H^2(\beta)$, which is invariant under
$M_z$. Let $T$ be the restriction of the $m$-tuple $M_z$ to
$H^2(\beta)_s$. Then $T$ is a spherical $m$-tuple (see Theorem 2.12 below), but the dimension
of $\ker(D_{T^*})$ is greater than one. This gives an example of an
$m$-tuple of operators of multiplication by the co-ordinate
functions $z_1, \dots, z_m$ on a Hilbert space of \textit{scalar}
power series in $z_1, \dots, z_m$, which does not satisfy the
equivalent conditions (1)-(3) of Theorem \ref{th:spher-tpl}. }
\end{example}

\begin{example}
{\rm Here we show that the existence of a cyclic vector for a
commuting spherical operator $m$-tuple $T$ also does not imply the
above conditions (1)-(3). Namely, take any integer
%
%
$\ell>m-\frac 12$.
Consider the Sobolev
space
$\mathcal H= W^{\ell,2}(\partial \mathbb{B})$; we refer to
\cite{Hebey} for a definition. We will need
also the dual space
$\mathcal H'= W^{-\ell,2}(\partial \mathbb{B})$;
its elements are complex-valued distributions, defined on the
unit sphere $\partial \mathbb{B}$. Both spaces
are Hilbert, and infinitely differential functions
are dense both in $\mathcal H'$ and in $\mathcal H$.
The pairing between $\mathcal H$ and $\mathcal H'$
is a continuation of the $L^2$ pairing
$\langle f, g\rangle = \int_{\partial \mathbb{B}} f\bar g$, defined
for $C^\infty$ functions.

Let $T$ be the multiplication tuple $M_z$ on $\mathcal H'$.
Since the spaces $\mathcal H$ and $\mathcal H'$
and their norms are invariant under unitary
rotations in ${\mathbb C}^n$, $T$ is a
(strongly) spherical tuple.

It is easy to see that $T$ is not unitarily equivalent to a spherical $m$-variable weighted shift.
Indeed, if $S$ denotes the $m$-tuple of multiplication by $\overline{z}$
on $\mathcal H'$ then $\sum_{j=1}^m T_jS_j = I.$ It follows that
$\mbox{ker}(D_{T^*})=\{0\},$ and
hence $T$ cannot be unitarily equivalent to a weighted shift.


Nevertheless, $T$ has a cyclic vector. Indeed, choose any  dense sequence $\{a_{n}\}$ of points on
$\partial \mathbb{B}$ such that their first coordinates $z_1(a_{n})$
are all distinct. The adjoint tuple to $T$ coincides with the multiplication tuple $M_z$, acting on $\mathcal H$.
By the Sobolev embedding theorem, $\mathcal H$ is continuously embedded
into $C(\partial \mathbb{B})$.
Hence for any sequence $\{c_n\}$ in $\ell^1$, the linear functional
$$
\psi(f)\defin \sum_n c_n f(a_{n})
$$
is bounded on $\mathcal H$ and therefore is an element of $\mathcal H'$.
We assert that if the sequence $\{c_n\}$ does not vanish and decays sufficiently fast, then
the vector $\psi\in \mathcal H'$ is cyclic for $M_{z_1}$ and therefore for the whole tuple $T$.

Indeed, suppose that some function $f\in \mathcal H$ satisfies
$$
\psi((\lambda - M_{z_1})^{-1}f)=\sum_n \frac {c_n
f(a_n)}{\lambda-z_1(a_n)}=0
$$
for any $\lambda$ with $|\lambda|>1.$ Suppose
that $c_n\ne 0$ for all $n,$ and $\sum_n n^{-2}\log |c_n|=-\infty$.
Since the points $z_1(a_{n})$ are all distinct,
it follows from a theorem by Sibilev \cite{Sib} that $c_n
f(a_n)=0$ for all $n$, which implies that $f$ is zero. Hence $\psi$
is cyclic for the operator $M_{z_1}$ on $\mathcal H'$.

We remark that a similar construction of a cyclic vector for a family of normal
operators is given in \cite{RossW}.
}
\end{example}

Before proving Theorems \ref{th:w-sh} and \ref{th:spher-tpl}, we need several lemmas.

\begin{lemma}
\label{lm:irr-repr}
Let $L$ be a finite-dimensional Hilbert space and let $\pi :
\mathcal U(m) \rar B(L)$ be an irreducible unitary representation
with respect to two unitary structures defined by scalar
products $\inp{\cdot}{\cdot}_1$ and $\inp{\cdot}{\cdot}_2$ on $L.$
Then there is a constant $\gamma > 0$ such that $$\inp{x}{y}_2=
\gamma \inp{x}{y}_1~(x, y \in L).$$
\end{lemma}

\begin{proof} By the Riesz Representation Theorem, there exists a
positive operator $A$ on $L$ such that
$\inp{x}{y}_2=\inp{Ax}{y}_1$ for every $x, y \in L.$ Since $L$ is
finite-dimensional and $A$ is positive, the point-spectrum of $A$
is a non-empty finite subset of $(0,+\infty)$. Let $\gamma$ be the
minimal eigenvalue of $A.$ We claim that $\mbox{ker}(A-\gamma
I)=L.$

Since $A-\gamma I$ is a nonnegative operator, for $x \in L,$ one
has \beq \label{eq (2.6)} \inp{x}{x}_2=\gamma
\inp{x}{x}_1~\mbox{iff}~\inp{(A-\gamma I)x}{x}_1=0~\mbox{iff}~x
\in \mbox{ker}(A-\gamma I).\eeq Let $U \in \mathcal U(m)$ and $x
\in \mbox{ker}(A-\gamma I).$ By assumption, $\pi(U)$ preserves
both scalar products, and hence by (\ref{eq (2.6)}),
$$\inp{\pi(U)x}{\pi(U)x}_2=\inp{x}{x}_2=\gamma \inp{x}{x}_1=\gamma
\inp{\pi(U)x}{\pi(U)x}_1.$$ It follows that $\mbox{ker}(A-\gamma I)$
is invariant under $\pi(U).$ Since $\pi(U)^*=\pi(U^{-1}),$
$\mbox{ker}(A-\gamma I)$ is indeed a reducing subspace for $\pi(U).$
Since $\mbox{ker}(A-\gamma I) \neq \{0\}$ and $\pi$ is irreducible
by assumption, we must have $\mbox{ker}(A-\gamma I)=L.$ Thus the
claim stands verified. The desired conclusion now follows from
(\ref{eq (2.6)}) and the polarization identity.
\end{proof}

We also need an analogue of this lemma for reducible
representations.

\begin{lemma}
\label{lm:repr} Let $L$ be a finite-dimensional Hilbert space and
let $\pi : \mathcal U(m) \rar B(L)$ be an unitary representation
with respect to a unitary structure defined by a scalar product
$\inp{\cdot}{\cdot}_1$. Let $L=L_1\oplus L_2 \oplus \dots \oplus
L_k$ be the corresponding decomposition of $L$ into irreducible
subspaces $L_j$ and suppose these subspaces are of distinct
dimensions. Suppose that we are given another semidefinite
sesquilinear product $\inp{\cdot}{\cdot}_2$ on $L$, which is
invariant with respect to $\pi$: $\inp{\pi(U)x}{\pi(U)y}_2=\inp  x
y_2$ for all $x,y\in L$ and all $U\in \mathcal U(m)$. Then there
are nonnegative constants $\bbeta_1, \dots, \bbeta_k$ such that
the following statements hold:

\textbf{(1)}
$\inp{x}{y}_2= \bbeta_j \inp{x}{y}_1~(x, y \in L_j)$;

\textbf{(2)}
$\inp{x}{y}_2= 0$ if $x\in L_p$, $y\in L_r$, $p\ne r$.

\end{lemma}

\begin{proof} Similarly to the previous proof,
there is a nonnegative operator $A$ on $L$ such that
$\inp{x}{y}_2=\inp{Ax}{y}_1$ for every $x, y \in L$.
By the assumption, one has a decomposition
$\pi=\pi_1\oplus \pi_2\oplus \dots \oplus \pi_k$, where
$\pi_j: \mathcal U(m) \rar B(L_j)$ are irreducible representations.
We obtain assertion (1) by applying Lemma \ref{lm:irr-repr} to
representations $\pi_j$ (if
the product $\inp{\cdot}{\cdot}_2$
is not definite, one can apply Lemma \ref{lm:irr-repr} to positive definite
products $\inp{\cdot}{\cdot}_1$ and $\inp{x}{y}_3=\inp{x}{y}_1+\inp{x}{y}_2$).
To see (2), note that $\pi_j$ are all inequivalent representations
and apply \cite[Corollary 2.21]{Se}.
\end{proof}

Next lemma will be crucial in the proof of Theorem \ref{th:spher-tpl}.
\begin{lemma}
\label{lm:gen-form} Let $T$ be a commuting, bounded spherical
operator $m$-tuple on $\mathcal H$. Suppose that $\ker(D_{T^*})$ is
one-dimensional and is spanned by a vector $e\in \mathcal H$.
Suppose that $e$ is cyclic for $T$. Then there is sequence of
positive weights $\{\bbeta_k\}_{k\ge 0}$ such that for any
polynomial $p\in \mathbb{C}[z_1,\dots, z_m]$, \bq \label{nrm-pol}
\|p(T)e\|^2=\sum_{k=0}^{\deg p}
\,\bbeta_k\|p_{[k]}\|^2_{L^2(\partial \mathbb B)} \eq where
$\|p\|^2_{L^2(\partial \mathbb B)}=\int_{\partial \mathbb B}|p(z)|^2
d \sigma(z)$ for the surface area measure $\sigma$ on the unit
sphere $\partial \mathbb B.$ The sequence $\{\bbeta_k\}$ is defined
uniquely.
\end{lemma}

\begin{proof} 
Notice first that $\ker D_{T^*}=(T_1 \cH + \dots + T_m\cH)^\perp$ is
invariant under the action of $\mathcal U(m)$. Hence for any $U$ in
$\mathcal U(m)$, there is a scalar constant $\zeta(U)$,
$|\zeta(U)|=1$, such that $\Gamma(U)e=\zeta(U)e$.

Fix a positive integer $N$, and denote by $H_N$ the space of
polynomials in $\mathbb C[z]$ of degree less or equal to $N$. Clearly, $H_N$ is a closed
subspace of $L^2(\partial \mathbb B)$; the corresponding scalar product
will be denoted as $\inp{\cdot }{\cdot }_1$.
Define a second semidefinite sesquilinear product on $H_N$ by
$$
\inp{p}{q}_2=\inp{p(T)e}{q(T)e}_{\mathcal H}.
$$
Both products are invariant under the action of $\mathcal U(m)$. Indeed,
$p(T_U)e=\zeta(U)\Gamma(U)^{-1}p(T)e$ for all $p\in \mathbb C[z]$ and
$U\in \mathcal U(m)$. Hence
\begin{align*}
\inp{p(Uz)}{q(Uz)}_2 & =\inp{p(T_U)e}{q(T_U)e}_{\mathcal H} \\ &=
\inp{\zeta(U)\Gamma(U)^{-1} p(T)e}{\zeta(U) \Gamma(U)^{-1}q(T)e}_{\mathcal H}
=\inp{p}{q}_2
\end{align*}
for all $p,q\in H_N$. It follows from \cite[pg. 175]{Se} that the
decomposition of $(H_N, \inp{\cdot }{\cdot }_1)$ into irreducible
subspaces with respect to the action of $\mathcal U(m)$ on $H_N$ is
given by $H_N=\Hom(0)\oplus \Hom(1) \oplus \dots \oplus
\Hom(N)$.
This fact and Lemma \ref{lm:repr} imply formula \eqref{nrm-pol} for
some nonnegative constants $\bbeta_0, \dots, \bbeta_N$. If a
constant $\bbeta_j$ were zero, it would follow that $p(T)e=0$ for
any homogeneous polynomial $p\in \Hom(j)$, which would imply that
$p(T)e=0$ for all $p\in \Hom(k,0)$ with $k>j$. Since $e$ is cyclic,
this would imply that $\mathcal H$ is finite dimensional, which
gives a contradiction.

Since $N$ is arbitrary, the statement of Lemma follows.
\end{proof}

\begin{proof}[Proof of Theorem \ref{th:w-sh}]
First of all, we mention that $\langle z^n, z^k\rangle_{L^2(\partial
\mathbb B)}=0$ for any distinct multi-indices $n,k\in \mathbb N^m$
see \cite[formula (1.21), page 13]{Z}. So the functions $z^n$, $n\in
\mathbb N^m$ form an orthogonal sequence in $L^2(\partial \mathbb
B)$. It follows that the norm, defined by \eqref{H2-beta-L2B}, is an
$H^2(\beta)$ norm for certain multi-sequence $\beta_n$. It is clear
that the multiplication tuple $M_z$ on the Hilbert space with the
norm \eqref{H2-beta-L2B} is spherical. This gives the ``if'' part of
the first statement.

Conversely, for each multiplication tuple $M_z$, the space
$\ker(D_{M^*_z})$ is one-dimensional and is spanned by the formal
power series $1$. So we can apply Lemma \ref{lm:gen-form} to get the
``only if'' part of the first statement.

Finally, one can make use of \eqref{H2-beta-L2B} and of the formula
\bq \label{norm-B} \int_{\partial \mathbb B}|z^n|^2 d
\sigma(z)=\frac{(m-1)!\, n!}{(m-1+|n|)!}~ (n \in \mathbb N^m), \eq
(see \cite[Lemma 1.11]{Z}) to derive the expressions \eqref{w-i-n}
and \eqref{beta-n} for $w^{(i)}_n$ and $\bbeta_n$ respectively.
\end{proof}

\begin{proof}[Proof of Theorem \ref{th:spher-tpl}]
The equivalence of (2) and (3) has been noted already. If (3) holds,
then $\ker(D_{T^*})$ is one-dimensional and is spanned by the image
in $\mathcal H$ of the formal power series $1$ under the unitary
equivalence. This implies (1). Finally, suppose that (1) holds, and
let $e$ be a unit vector that spans $\ker(D_{T^*})$. Then it follows
from Lemma \ref{lm:gen-form} that there is a sequence $\bbeta_0,
\bbeta_1, \bbeta_2, \dots$ such the map $p\mapsto p(T)e$, $p\in
\mathbb{C}[z]$ extends to a unitary map from $H^2(\beta)$ to
$\mathcal H$, which intertwines $T$ with $M_z$.
\end{proof}

Let $\Lambda\subset \mathbb{Z}^m_+$ be a set of multi-indices. In
what follows, we will say that $\Lambda$ is \textit{inductive} if
for any $n\in \Lambda$, the multi-indices $n+\epsilon_j$ are also in
$\Lambda $ for $j=1, \dots, m$.

\begin{theorem}
\label{th:toral} Let $\Omega$ be a Reinhardt domain in
$\mathbb{C}^m$ such that $0 \in \Omega$.
 Let $\mathscr H$ be a $M_z$-invariant RKHS of functions on
$\Omega$ such that  $\mathscr H\subset
\Hol(\Omega)$, the inclusion being continuous. Let $\kappa(z, w)~(z,
w \in \Omega)$ denote the reproducing kernel of $\mathscr H.$

Then the following statements are equivalent.

\begin{enumerate}

\item For every $\zeta \in \mathbb T^m,$
\beq \label{t-kernel}
\kappa(\zeta \cdot z, \zeta \cdot w)=\kappa(z, w)~(z, w
\in \Omega),
\eeq
where $\zeta \cdot z = (\zeta_1 z_1, \cdots, \zeta_m z_m) \in \mathbb C^m.$

\item For every $\zeta \in \mathbb T^m,$ $f(\zeta \cdot) \in \mathscr H$ whenever $f \in \mathscr H,$ and
\beqn \inp{f(\zeta \cdot)}{g(\zeta \cdot)}=\inp{f}{g}~(f, g \in \mathscr H).\eeqn

\item
There exist a multi-sequence $\{\beta_n\}_{k\in \mathbb Z_+^m}$ and
 an inductive set $\Lambda \subset \mathbb Z^m_+$ such that
$\mathscr H= H^2(\beta)_\Lambda$,
where
\[
H^2(\beta)_\Lambda = \big\{f\in H^2(\beta): D^n f(0)=0\quad \text{for
all}\; n  \in \mathbb Z^m_+, \; n \notin \Lambda \big\}.
\]

\item
There exists an inductive set $\Lambda' \subset \mathbb Z^m_+$ such that
the functions $z^n$, $n\in \Lambda'$, are contained in $\mathscr H$ and form there
an orthogonal basis.

\item
There exist an inductive set $\La'' \subset \mathbb Z^m_+$
and a family $\{\al_n\}_{n \in \La''}$ of positive numbers such that
\bq
\label{expr-kap}
\kappa(z, w)=\sum_{n \in \La''}\al_n z^n
\overline{w}^n~(z, w \in \Omega).
\eq
\end{enumerate}
Moreover, if (1)--(5) hold, then $\La=\La'=\La''$.
\end{theorem}

In (3), in the equality $\mathscr H= H^2(\beta)_\Lambda$ we
identify analytic functions in $\Omega$ with the corresponding formal power series
centered at the origin. This equality means
that these two Hilbert spaces consist of the same functions and the
norms in these two spaces are identical.

\begin{theorem}
\label{th:spherical} Let $\mathscr H$ be a
$M_z$-invariant RKHS of functions on
$\mathbb B_R$ in $\mathbb C^m$. Suppose $\mathscr H\subset
\Hol(\mathbb B_R)$, the inclusion being
continuous. Let $\kappa(z, w)~(z, w \in \mathbb B_R)$ denote the
reproducing kernel of $\mathscr H.$

Then the following statements are equivalent.

\begin{enumerate}

\item For every $U \in \mathcal U(m),$
\beq \label{s-kernel} \kappa(Uz, Uw)=\kappa(z, w)~(z, w
\in \mathbb B_R). \eeq
\item For every $U \in \mathcal U(m),$ $f(U \cdot) \in \mathscr H$ whenever $f \in \mathscr H,$ and
\beqn \inp{f(U \cdot)}{g(U \cdot)}=\inp{f}{g}~(f, g \in \mathscr H).\eeqn

\item
There exist $s \in \mathbb Z_+$ and a
scalar sequence $\{\bbeta_k\}_{k\in \mathbb N}$ such that
$\mathscr H= H^2(\beta)_s$, where
the multi-sequence $\beta$ is given by \eqref{beta-n} and
\[
H^2(\beta)_s=\big\{f\in H^2(\beta): D^n f(0)=0\quad \text{for
all}\; n\in \mathbb Z^m_+, \;|n|<s\big\}.
\]
\end{enumerate}
If any of the conditions (1) -- (3) holds, then
$M_z$ is a strongly spherical tuple.

\end{theorem}

\begin{remark}
Some statements  close to the above Theorem~\ref{th:spherical} are given in the beginning of Section~4 of
the paper \cite{G-H-X} by Guo, Hu and Xu, though they do not discuss the continuity of the representations $\Gamma$.
As follows from their discussion, the spaces $H^2(\beta)_s$ are defined uniquely by their
generating function $F(t)$, analytic in the disc $|t|<  R^2 $ in the complex plane, such that
\[
\kappa(z,w)=F\big(\inp z w\big)~(z,w\in \mathbb B_R).
\]
Such representation always exists, all the coefficients $a_n$ in the expansion
$F(t)=\sum_{k=s}^\infty a_k t^k$ are positive and are given by
\bq
\label{a-k}
a_k = \frac {(m-1+k)!} {(m-1)! \,k!}\frac{1}{\bbeta^2_k},~k\ge s
\eq
(it follows from \eqref{rkernel} and \eqref{beta-n}).
\end{remark}

\begin{lemma}
\label{le2.11} Let $G$ be a subgroup of the group $GL_m(\mathbb C)$
of invertible, complex $m \times m$ matrices and let $\Omega$ be a
$G$-invariant (that is, $\mathfrak g z \in \Omega$ whenever
$\mathfrak g \in G$ and $z \in \Omega$) domain in $\mathbb{C}^m$
such that $0 \in \Omega.$ Let $\mathscr H$ be a $M_z$-invariant RKHS
of functions on $\Omega$ such that  $\mathscr H \subset
\Hol(\Omega)$, the inclusion being continuous. Let $\kappa(z, w)~(z,
w \in \Omega)$ denote the reproducing kernel of $\mathscr H.$
Let $M_z=(M_{z_1}, \cdots, M_{z_m})$ be the bounded $m$-tuple of
multiplication by the co-ordinate functions $z_1, \cdots, z_m.$ Then
the following statements are equivalent:
\begin{enumerate}

\item For every $\mathfrak g \in G,$
\beq \label{2.9} \kappa(\mathfrak g z, \mathfrak g w)=\kappa(z,
w)~(z, w \in \Omega). \eeq

\item For every $\mathfrak g \in G,$ $f(\mathfrak g \cdot) \in \mathscr H$ whenever $f \in \mathscr H,$ and
\beqn
\inp{f(\mathfrak g \cdot)}{h(\mathfrak g \cdot)}=\inp{f}{h}~(f, h \in \mathscr H).
\eeqn
\end{enumerate}
If this happens then the representation $\Gamma : G \rar
B(\mathscr H)$ of $G$ on $\mathscr H$ given by
\beq
\Gamma(\mathfrak
g)f(z)=f(\mathfrak g z)~(z \in \Omega, \mathfrak g \in G)
\eeq
is strongly continuous, unitary and
satisfies $\Gamma(\mathfrak g)M_{z_j}=M_{(\mathfrak g
z)_j}\Gamma(\mathfrak g)~(j=1, \cdots, m).$
In particular, if $G=\mathcal{U}(m)$, then $M_z$ is strongly spherical.
\end{lemma}

\begin{proof}
(1) implies (2):
Suppose that (1) holds. Set \beqn \Gamma(\mathfrak g)\kappa(\cdot, w)= \kappa(\cdot,
\mathfrak g^{-1}(w))~(w \in \Omega, \mathfrak g \in G). \eeqn
We check that $\Gamma$ extends to a unitary representation of $G$
on $\mathscr H.$ By the reproducing property of $\kappa$ and \eqref{2.9}, \beqn \inp{\Gamma(\mathfrak g)\kappa(\cdot,
z)}{\Gamma(\mathfrak g)\kappa(\cdot, w)}
= \inp{\kappa(\cdot,
z)}{\kappa(\cdot, w)}. \eeqn
Since $\bigvee\{\kappa(\cdot, w): w \in \Omega
\}=\mathscr H,$ $\Gamma(\mathfrak g)$ extends isometrically to the entire
$\mathscr H.$ Since $\mathfrak g (\Omega)=\Omega,$ $\Gamma(\mathfrak g)$ is
surjective, and hence unitary. Finally, since
$\Gamma(\mathfrak g)^*=\Gamma(\mathfrak g^{-1}),$ it follows that
\beqn
\Gamma(\mathfrak g)f(z) =
\inp{\Gamma(\mathfrak g)f}{\kappa(\cdot, z)}
= \inp{f}{\Gamma(\mathfrak g^{-1})\kappa(\cdot, z)}
= f(\mathfrak g z)
\eeqn
for any $z \in \Omega$ and any $f \in
\mathscr H.$

(2) implies (1): Assume that (2) is true. By the uniqueness of the
reproducing kernel, it suffices to check that $\kappa(\mathfrak g z,
\mathfrak g w)$ is a reproducing kernel for $\mathscr H$ for every
$\mathfrak g \in G.$ However,
\[
\inp{f}{\kappa(\mathfrak g \cdot,
\mathfrak g w)}=\inp{f(\mathfrak g^{-1}\cdot)}{\kappa(\cdot,
\mathfrak g w)}= f(w)~(w \in \Omega),
\]
which gives (1).

The fact that $\Gamma$ is a unitary representation of $G$ on
$\mathscr H$ follows from (2).
It follows from the closed graph theorem that
the operators $M_{z_j}$ are bounded.
Notice that by Hartogs' separate analyticity theorem \cite{Kr},
$\kappa(z, \overline{w})$ is
holomorphic in $z,w$, and it follows that
the map $w \rightarrow \kappa(w, w)$ is continuous.
Since
$\|\kappa(\cdot, w)-\kappa(\cdot, w_0)\|^2=
\kappa(w, w)+ \kappa(w_0, w_0) - 2\operatorname{Re}\kappa(w, w_0)$,
the function $w\mapsto \kappa(\cdot, w)\in \mathcal H$ is norm continuous.
Therefore $\Gamma (\mathfrak g) \kappa(\cdot, w)$ depends continuously on
$\mathfrak g$ for any $w$. Since the reproducing kernels are complete,
$\Gamma$ is strongly continuous. The remaining part is a routine
verification.
\end{proof}
\begin{remark} We are particularly interested in the
subgroups $\mathcal U \mathcal D(m)$ and $\mathcal U(m)$ of
$GL_m(\mathbb C),$ where $\mathcal U \mathcal D(m)$ denotes the
subgroup of unitary diagonal $m \times m$ matrices.
\end{remark}


\begin{proof}[Proof of Theorem \ref{th:toral}]
By Lemma \ref{le2.11}, (1) and (2) are equivalent.
It is clear that (3) and (4) are equivalent,
and the corresponding sets $\La$ and $\La'$ coincide
whenever (3) and (4) hold.
It is also clear that (3) implies (2).

(2) implies (4). Assume that (2) holds.
Define the set $\La'\subset \mathbb N^m$ by
\[
\La'=\big\{
n_0\in \mathbb N^m: \exists f=\sum a_n z^n \in \mathscr H:
a_{n_0}\ne 0
\big\}.
\]
Since $\mathscr H$ is $M_z$-invariant, $\La'$ is inductive.

Put $S(t)f(z)=f(e^{it}z)$, $t\in \mathbb R^m$,
where $e^{it}z=(e^{it_1}z_1, \dots, e^{it_m}z_m)$.
By applying Lemma~\ref{le2.11}, we get that $S$
is a unitary strongly continuous $m$-parameter group.
Given any function $f(z)=\sum a_n z^n \in \mathscr H$
and any $n_0\in\mathbb N^m$ such that $a_{n_0}\ne 0$,
we notice that
\[
a_{n_0}z^{n_0}=\frac 1{(2\pi)^m}\, \int_{[0,2\pi]^m} e^{-in_0t} S(t) f\, dt.
\]
(The integral is understood in the Bochner sense. The equality is true because it holds
pointwise for any $z\in \Omega$.)
It follows that for any $n_0\in \La'$, $z^{n_0}\in \mathscr H$.

%
%

Now take any $p, q \in \mathbb N^m$ such that $p \neq q.$ Then for some $1
\leq j \leq m,$ $p_j \neq q_j.$
Let $\zeta=w \varepsilon_j + \sum_{i \neq j} \epsilon_i$, where
$w \in \mathbb T$.  Then $\inp{z^p}{z^q}=\inp{\zeta
z^p}{\zeta z^q}=w^{p_j - q_j}\inp{z^p}{z^q},$ which is
possible for all $w \in \mathbb T$ only if $\inp{z^p}{z^q}=0$. We have checked
that the functions $z^n$, $n\in \Lambda'$ form an orthogonal sequence in
$\mathscr H$.
Any $f\in \mathscr H$ has a Taylor series representation $f(z)=\sum_{n\in \La'} a_n z^n$, which
converges weakly in $\mathscr H$. Therefore the sequence
$\{z^n\}_{n\in \Lambda'}$ is in fact an orthogonal basis in $\mathscr H$.

%

Given any orthonormal basis $\{\phi_k\}_{k\in \mathscr K}$ in $\mathscr H$,
the reproducing kernel of $\mathscr H$ can be expressed by the well-known formula
$\kappa(z,w)=\sum_{k\in \mathscr K} \overline{\phi_k(w)}\phi_k(z)$.
It follows that (3) implies (5) (with $\La''=\La$).

It is immediate that (5) implies (1), which concludes the proof of the fact that conditions (1)--(5) are
all equivalent. It also has been shown already that if (1)--(5) are fulfilled,
then $\La=\La'=\La''$.
\end{proof}

\begin{proof}[Proof of Theorem \ref{th:spherical}]
By Lemma \ref{le2.11}, (1) is equivalent to (2).
It is clear that (3) implies (2).
It remains to prove that (2) implies (3).
Suppose that (2) holds.
Then we can apply Theorem \ref{th:toral} and deduce that
$\mathscr H=H^2(\be)_\La$ for an inductive set $\La$.
Let $s=\min\{|n|: n\in \La\}$, then the intersection
$R$ of $\mathscr H$ with the space $\Hom(s)$ of analytic
homogeneous polynomials of order $s$ is non-zero, and the
group $\mathcal U(m)$ acts on $R$. Since the action of
$\mathcal U(m)$ on $\Hom(s)$ is irreducible (we already have
used it in Lemma~\ref{lm:gen-form}), it follows that $R=\Hom(s)$. Since
$\La$ is inductive, $\La=\{n\in \mathbb N^m: |n|\ge s\}$,
which gives (3).

By Lemma \ref{le2.11},
if any of the equivalent conditions (1)--(3) holds, then the tuple $M_z$
consists of bounded operators. Now (3) implies that
$M_z$ is strongly spherical.
\end{proof}

\section{Spectral Theory for Multi-shifts}

For a masterful exposition of various notions of invertibility,
Fredholmness and multi-parameter spectral theory, the reader is
referred to \cite{Cu}. For $T \in B(\mathcal H),$ we reserve the
symbols $\sigma(T), \sigma_p(T)$, $\sigma_{ap}(T), \sigma_{e}(T)$ for
the Taylor spectrum, point-spectrum, approximate-point spectrum, essential
spectrum of $T$ respectively. It is well known that the spectral
mapping theorem for polynomial mappings holds for both the Taylor
and the approximate-point spectra. Except the point-spectrum, all spectra mentioned above are
always non-empty.

Given a commuting $m$-tuple $T = (T_1, \cdots, T_m)$ of operators on
${\mathcal H},$  set
\bq
\label{QT}
Q_T(X) \mathrel{\mathop:}= \sum_{i=1}^mT^*_iXT_i~(X \in B(\mathcal H)).
\eq
We define inductively
$Q^0_T(I)=I$ and
$Q^k_T(I)=Q_T\big(Q^{k-1}_T(I)\big)$ for $k\ge 1$.
Then we have
\bq
\label{QkT}
Q^k_T(I)=\sum_{|\alpha|=k}\,\frac{k!}{\alpha!} \, {T^*}^{\alpha}T^{\alpha}.
\eq

\begin{lemma}
Let $T$ be a spherical commuting, bounded $m$-variable weighted
shift with respect to an orthonormal basis $\{e_n\}_{n \in \mathbb
N^m}.$ Let $T_{\delta} : \{\delta_{k}\}_{k \in \mathbb N}$ be the
(one-variable) shift associated with $T$ with respect to an
orthonormal basis $\{f_k\}_{k \in \mathbb N}.$ Then
\beq
\label{id}
\inp{Q^k_T(I)e_n}{e_n}=\|T^k_{\delta}f_{|n|}\|^2~ (k \in \mathbb N,
~n \in \mathbb N^m).
\eeq
\end{lemma}
\begin{proof}
It is easy to see that
the operator $Q^k_T(I)$ is diagonal with respect to the basis
$\{e_n\}$, and
%
%
\beq
\label{Qken}
Q^k_T(I)e_n =
\delta^2_{|n|} \delta^2_{|n| +1} \cdots \delta^2_{|n| + k-1} e_n~ (k \in \mathbb N, ~n \in \mathbb N^m).
\eeq
The desired conclusion is now immediate.
\end{proof}

\begin{proposition}
\label{sp-rad} Let $T$ be a spherical commuting, bounded
$m$-variable weighted shift with respect to the orthonormal basis
$\{e_n\}_{n \in \mathbb N^m}.$ Let $T_{\delta} : \{\delta_{k}\}_{k
\in \mathbb N}$ be the shift associated with $T$ with respect to the
orthonormal basis $\{f_k\}_{k \in \mathbb N}.$ Then the
geometric spectral
radius $r(T):=\sup \{\|z\|_2 : z \in \sigma(T)\}$ of $T$ is equal to
the spectral radius of $T_{\delta}$.
\end{proposition}
\begin{proof}
By \cite[Theorem 1]{M-S} and \cite[Theorem 1]{C-Z}, the geometric spectral
radius $R$ of $T$ is given by
\beqn
R=\lim_{k \rar \infty}
\big\|Q_T^k(I)\big\|^{\frac{1}{2k}}.
\eeqn It is easy to see that the orthogonal basis $\{e_n\}_{n \in
\N^m}$ diagonalizes the positive operator
\linebreak                               
$Q^k_T(I)$. Also, by (\ref{id}),
$\inp{Q^k_T(I)e_n}{e_n}=\|T^k_{\delta}f_{|n|}\|^2$ for every $k \in
\mathbb N$ and $n \in \mathbb N^m.$ It follows that
\beqn
\lim_{k \rar \infty}\big\| Q^k_T(I) \big\|^{1/2k}
=
\sup_{k \geq 0}\|T^k_{\delta}\|^{1/k} = r(T_\delta),
\eeqn
by the well-known general formula for the spectral radius of a linear operator.
\end{proof}

Let ${\mathcal C}(\mathcal H)$ denote the norm-closed
ideal of compact operators on $\mathcal H.$
Since $B(\mathcal H) / {\mathcal C}(\mathcal H)$ is a unital
$C^*$-algebra, the {\it Calkin algebra}, there exist a Hilbert space
$\mathcal K$ and an injective unital $*$-representation $\pi :
B(\mathcal H) / {\mathcal C}(\mathcal H) \rar B(\mathcal K)$
\cite[Chapter VIII]{Co-0}. In particular, $\pi \circ q : B(\mathcal
H) \rar B(\mathcal K)$ is a unital $*$-representation, where $q :
B(\mathcal H) \rar B(\mathcal H)/ {\mathcal C}(\mathcal H)$ stands
for the {\it quotient} ({\it Calkin}) {\it map}. Set $\pi \circ
q(T):= (\pi \circ q(T_1), \cdots, \pi \circ q(T_m)).$

We recall that a tuple $T=(T_1,\dots,T_m)$ is called \textit{essentially normal}
if all commutators $[T_j,T_k]$ and $[T_j^*,T_k]$, $j,k=1,\dots,m$ are compact.
The following (known) version of the Fuglede--Putman commutativity theorem
holds: given operators $A$ and $N$ on a Hilbert space $H$,
if $N$ is essentially normal and the commutator $[A,N]$ is compact,
then the commutator $[A,N^*]$ also is compact. This follows
by applying the classical Fuglede--Putman theorem
to operators $\pi\circ q(N)$ and $\pi\circ q(A)$.
We refer to \cite{Wei} for an additional information.

It follows that a commutative tuple $T$ is essentially normal
whenever $[T^*_j, T_j]$ are compact for $j=1,\dots, m$.

\begin{remark}
\label{le4.5} Let $T : \{w^{(i)}_n\}$ be a bounded spherical $m$-variable
weighted shift and let $T_{\delta} : \{\delta_k\}_{k \in \mathbb N}$
be the one-variable shift associated with $T.$
As it follows from \cite[Corollary 4.4]{G-H-X}, $T$ is essentially normal if and only if
$T_{\delta}$ is essentially normal if and only if $\lim_{ k \rar
\infty} \left(\delta^2_{k} - \delta^2_{k-1}\right)=0$
(observe that by~\eqref{a-k}, $\frac{a_k}{a_{k+1}}=\frac{k+1}{k+m}\frac{\bbeta_{k+1}^2}{\bbeta_k^2}$).
\end{remark}
%

The main theorem of this section describes several spectral parts of spherical $m$-shifts.
\begin{theorem}
\label{th:3.1} Let $M_z$ be a bounded spherical multiplication $m$-tuple in
$H^2(\beta)$, so that
the norm in $H^2(\beta)$ is given by \eqref{H2-beta-L2B} for a certain
sequence $\bbeta_0, \bbeta_1, \bbeta_2, \cdots, $ of positive numbers.
%
%
%
Let $R(M_z), r(M_z), i(M_z)$ be given by
\beq
\label{eq-sp-rad}
R(M_z)= \lim_{j \rar \infty} {\sup_{k \geq
0} \sqrt[j]{\frac{{\bbeta_{k+j}}}{{\bbeta_k}}}},
\eeq
\beq
\label{rad-cgn} r(M_z) = \liminf_{j \rar \infty} \sqrt[j]{\bbeta_j},
\eeq
\beq \label{inn-rad} i(M_z)= \lim_{j \rar \infty} {\inf_{k \geq
0} \sqrt[j]{\frac{{\bbeta_{k+j}}}{{\bbeta_k}}}}.
\eeq
Then $i(M_z)\le r(M_z)\le R(M_z)$, and the following statements are true:
\begin{enumerate}

\item
The Taylor spectrum of $T$ is the closed ball $\overline{\mathbb
B}_{R(M_z)}$ in $\mathbb C^m.$

\item The ball $\mathbb B_{r(M_z)}$ in $\mathbb C^m$ is the largest open ball in which
all the power series in $H^2(\beta)$ converge.

\item Either
$\sigma_p(T^*)=\mathbb B_{r(M_z)}$ or $\sigma_p(T^*) = \overline{\mathbb
B}_{r(M_z)}.$

%

\item  $\sigma_{ap}(M_z)= \overline{A}_{i(M_z),\, R(M_z)}$,
where $\overline{A}_{i(M_z),\, R(M_z)}$ stands for the
closed ball shell
in $\mathbb C^m$ of inner-radius $i(M_z)$ and outer-radius $R(M_z)$.

\item If in addition, $\lim_{k \rar \infty} \frac{\bbeta_{k+1}}{\bbeta_k} -
\frac{\bbeta_{k}}{\bbeta_{k-1}}=0$ then the essential spectrum of
$M_z$ is the closed ball shell of inner-radius $\liminf_{k \rar
\infty} \frac{\bbeta_{k+1}}{\bbeta_k}$ and outer-radius $\limsup_{k
\rar \infty} \frac{\bbeta_{k+1}}{\bbeta_k}.$
\end{enumerate}
\end{theorem}
The first part of the Theorem \ref{th:3.1}  is obtained in
\cite[Theorem 4.5(1)]{G-H-X}, under the additional assumption of essential normality, by entirely different methods.
An upper estimate of the geometric joint spectral radius of
$M_z$ is given in \cite[Theorem 9.6]{Kapt}.

Note further that statement (5) of the theorem is precisely
\cite[Theorem 4.5(2)]{G-H-X}.
We can give a more general version of this statement.

\begin{lemma} \label{Pro4.1ess}
Let $T$ be an essentially normal spherical $m$-tuple. Then
the essential spectrum of $T$ is given by
\beqn
\sigma_e(T)=\left\{z \in \mathbb C^m : \|z\|^2_2 \in
\sigma_e(Q_T(I))\right\}.
\eeqn
\end{lemma}
\begin{proof}
We adapt the proof of \cite[Theorem 4.5(2)]{G-H-X} to the
present situation. Suppose $T$ is essentially normal. Equivalently,
$(q(T_1), \cdots, q(T_m))$ is a commuting normal $m$-tuple in the
Calkin Algebra. Let $\mathcal M$ be the maximal ideal space of the
commutative $C^*$-algebra $C^*(q(T))$ generated by $q(T_1), \cdots,
q(T_m).$
By \cite[Corollary 3.10]{Cu-0}, the essential spectrum of $T$ is
given by
\beqn
\sigma_e(T) = \{(\phi(q(T_1)), \cdots, \phi(q(T_m)) :
\phi \in \mathcal M\}.
\eeqn
If $\lambda \in \sigma_e(T)$ then for
some $\phi \in \mathcal M,$ \beqn q(Q_T(I) - \|\lambda\|^2 I) =
\sum_{j=1}^m q(T^*_j)q(T_j) -|\phi(q(T_j))|^2. \eeqn Clearly, $\phi$
annihilates $q(Q_T(I) - \|\lambda\|^2 I) \in C^*(q(T)).$ Thus
$q(Q_T(I) - \|\lambda\|^2 I)$ is not invertible, and hence
$\|\lambda\|^2 \in \sigma_e(Q_T(I)).$ Conversely, suppose
$\|\lambda\|^2_2 \in \sigma_e(Q_T(I))$ for some $\lambda \in \mathbb
C^m.$ Thus $q(Q_T(I) - \|\lambda\|^2_2 I)$ is not invertible in the
Calkin algebra, and hence in $C^*(q(T)).$ Thus there exists some
$\phi_{\lambda} \in \mathcal M$ annihilating $q(Q_T(I) -
\|\lambda\|^2_2 I).$ This gives $\|\lambda\|^2_2= \sum_{j=1}^m
|\phi_{\lambda}(q(T_j))|^2.$ On the other hand,
$(\phi_{\lambda}(q(T_1)), \cdots, \phi_{\lambda}(q(T_m)) \in
\sigma_e(T).$ By the spherical symmetry of the essential spectrum,
we must have $\lambda \in \sigma_e(T).$
\end{proof}


Let us pass to the proof of Theorem \ref{th:3.1}. It involves several lemmas
and propositions. The first lemma
is a multi-variable analog of a well-known fact about
the approximate point spectrum \cite[Proposition 13]{Sh}.

\begin{lemma} \label{lem3.5n}
Let $T$ be a commuting $m$-tuple of operators on a Hilbert space. Then
the approximate point-spectrum of $T$ is disjoint from the
open ball $\mathbb B_{m_\infty(T)},$ where
\bq
\label{m-infty}
m_\infty(T)=
\sup_{k\ge 1} \enspace
\inf_{h \in \mathcal
H,~ \|h\|=1 } \big\|Q_T^k(I)h\big\|^{\frac 1 {2k}}.
\eq
\end{lemma}

\begin{proof}
Take any $\lambda \in
\mathbb C^m$ such that $\|\lambda\|_2 < m_\infty(T)$.
Then there exist some $\mu>0$ and integer $k > 0$
such that $\|\lambda\|_2 < \mu$ and
$\|Q_T^k(I) h\|\ge \mu^{2k}\|h\|$ for any $h \in \mathcal H$.

Put $\la^\al=\la_1^{\al_1}\dots \la_m^{\al_m}$ for $\al\in\mathbb N^m$. Notice that
\[
\sum_{|\al|=k}\frac {k!}{\al!}\, |\la^\al|^2 =
\big(|\la_1|^2+\dots+|\la_m|^2\big)^k
=\|\la\|_2^{2k}.
\]
Hence, by the Cauchy-Schwarz inequality,
for any unit vector $h \in \mathcal H$,
\beqn
\sum_{|\al|=k} \frac {k!}{\al!} \, \|T^\al h- \la^\al h\|^2 & \geq &
 \sum_{|\al|=k} \frac {k!}{\al!} \, \big(\|T^\al h\|- |\la^\al|\big)^2 \\
&\geq & \sum_{|\al|=k} \frac {k!}{\al!} \, \|T^\al h\|^2
 -2\big(\sum_{|\al|=k} \frac {k!}{\al!} \, \|T^\al h\|^2\big)^{1/2} \|\lambda\|_2^k + \|\lambda\|_2^{2k} \\
& = &
\|Q_T^k(I) h\|-2 \|Q_T^k(I) h\|^{1/2} \|\lambda\|_2^k + \|\lambda\|_2^{2k}
\geq  (\mu^k -\|\la\|_2^k)^2 >0. 
\eeqn
Therefore $\la\notin \sigma_{ap}(T)$.
\end{proof}

\begin{proposition} \label{Pro4.1symm}
The Taylor spectrum, approximate point-spectrum,
point-spectrum, essential spectrum of a spherical $m$-tuple are spherically symmetric.
\end{proposition}

\begin{proof}
The spherical symmetry of Taylor spectrum and approximate point-spectrum follows immediately from the spectral mapping property for polynomials.
On the other hand, the spherical symmetry of the point spectrum follows from the definition.
We now check the assertion for the essential
spectrum. Let $T$ be our spherical $m$-tuple and let $\pi,$ $q,$ $\mathcal K$ be as in the discussion
following Proposition \ref{sp-rad}. One may deduce from
\cite[Theorem 6.2]{Cu} and spectral permanence for the Taylor
spectrum that $\sigma_e(T) = \sigma(\pi \circ q(T))$. Since the
Taylor spectrum of a spherical tuple has spherical symmetry, it now
suffices to show that $\pi \circ q(T)$ is spherical.
This follows from Remark \ref{rk1.2}(2).
\end{proof}

In the single variable case, the following result was obtained by R.
L. Kelley (refer to \cite{R}).
\begin{lemma}
\label{le3.2} Let $M_z$ be a bounded multiplication $m$-tuple in
$H^2(\beta).$ Then the Taylor spectrum of $M_z$ is connected.
\end{lemma}
\begin{proof}
By \cite[Corollary 3]{J-L}, $\sigma(M_z)$
has a poly-circular symmetry (that is, $\zeta\cdot w\in
\sigma(M_z)$ for any $w\in \sigma(M_z)$ and any $\zeta\in \mathbb T^m$).
%
%
Note that $0$ belongs to the point spectrum $\sigma_p(M^*_z)$ of
$M^*_z,$ and hence to the Taylor spectrum of $M_z$ in view of
$\sigma_p(M^*_z) \subseteq \sigma(M^*_z) =\{\bar{z} : z \in
\sigma(M_z)\}$ (for $z=(z_1,\dots,z_m)$, we put $\bar z=(\bar z_1,\dots,\bar z_m)$).
It suffices to check that $\sigma(M^*_z)$ is connected. Let $K_1$
be the connected component of $\sigma(M^*_z)$ containing $0$ and let
$K_2=\sigma(M^*_z) \setminus K_1.$ By the Shilov Idempotent Theorem
\cite[Application 5.24]{Cu}, there exist invariant subspaces
$\mathcal M_1, \mathcal M_2$ of $M^*_z$ such that $\mathcal H =
\mathcal M_1 \dotplus \mathcal M_2$ and $\sigma(M^*_z|_{M_i}) =K_i$
for $i=1, 2.$

Let $h \in \mbox{ker}(D_{S^*_k}),$ where $S_k:=(M^k_{z_1}, \cdots,
M^k_{z_m})$ for a positive integer $k.$ Then $h=x + y$ for $x \in
\mathcal M_1$ and $y \in \mathcal M_2.$ It follows that
$(M^*_{z_j})^k x=0$ and $(M^*_{z_j})^ky=0$ for all $j=1, \cdots, m.$
If $y$ is non-zero then $0 \in \sigma_p(S^*_k) \subseteq \sigma(S^*_k),$ and hence by the spectral mapping property, $0 \in \sigma(M^*_z|_{\mathcal M_2}).$
Since $0 \notin K_2,$ we must have $y =0.$ It follows that $\mathcal
M_1$ contains the dense linear manifold $\bigcup_{k \geq 1}
\mbox{ker}(D_{S^*_k}),$ and hence $\mathcal M_1 = \mathcal H.$ Thus
the Taylor spectrum of $M^*_z$ is equal to $K_1.$ In particular, the
Taylor spectrum of $M_z$ is connected.
\end{proof}

\begin{lemma}
\label{lem4.3}
\beqn
\lim_{j \rar \infty}
\sup_{k \geq 0}
\Big(
\frac{\bbeta_{k+j+1}}{\bbeta_{k+1}}
\Big)^{\frac 1j}\,
\bigg(
\frac{(k+2) \cdots
(k+j+1)}{(k+m+1) \cdots (k+j+m)}\bigg)^{\frac 1{2j}}
=\lim_{j \rar \infty} \sup_{k \geq
0} \bigg(
\frac{\bbeta_{k+j+1}}{\bbeta_{k+1}}
\bigg)^{\frac 1j},
\eeqn
\beqn
\lim_{j \rar \infty}
\inf_{k \geq 0}
\Big(
\frac{\bbeta_{k+j+1}}{\bbeta_{k+1}}
\Big)^{\frac 1j}\,
\bigg(
\frac{(k+2) \cdots
(k+j+1)}{(k+m+1) \cdots (k+j+m)}\bigg)^{\frac 1{2j}}
= \lim_{j \rar \infty} \inf_{k \geq 0}
\Big(
\frac{\bbeta_{k+j+1}}{\bbeta_{k+1}}
\Big)^{\frac 1j}.
\eeqn
\end{lemma}
\begin{proof}
For $k\ge 0$ and $j\ge 1$, put
\[
\rho_{kj}:=
\frac{(k+2) \ldots
(k+j+1)}{(k+m+1) \ldots (k+j+m)}
=
\frac{(k+2) \ldots
(k+m)}{(k+j+2) \ldots (k+j+m)}\, .
\]
It is easy to see that $\rho_{kj}$
is an increasing function of $k$ (for a fixed $j$).
Hence
$\rho_{0j}^{\frac 1 {2j}}\le \rho_{kj}^{\frac 1 {2j}}\le 1$ for all $k\ge 0$.
Since
$\rho_{0j}^{\frac 1 {2j}}\rar 1$ as $j\rar \infty$, both statements of the Lemma follow.
\end{proof}



\begin{proof} [Proof of Theorem \ref{th:3.1}]
(1): Suppose $M_z$ is spherical. We already recorded that the Taylor
spectrum of a spherical tuple has spherical symmetry.
By Lemma \ref{le3.2}, the Taylor spectrum of $M_z$ is connected.
It is easy
to see that the only bounded closed connected subsets of $\mathbb C^m$ with spherical
symmetry are balls and ball shells.
This follows from the fact that the unitary
group $\mathcal U(m)$ acts transitively on any sphere. Since $0$ belongs to the spectrum
$\sigma(M_z),$ it must be a ball centered at the origin. The formula
for the spectral radius of $M_z$ now follows from Proposition
\ref{sp-rad} and the known formula for the
spectral radius of $T_\delta$ \cite{Sh}.

(2): Let $w \in \mathbb B_{r(M_z)}.$ We claim that any power series
in $H^2(\beta)$ converges absolutely at $w$. It suffices to check that $w$ belongs to the point
spectrum $\sigma_p(M^*_z)$ of $T^*$, or, equivalently, $\sum_{n \geq
0} {|w^n|^2}/{\|z^n\|^2_\beta} < \infty$ (see \cite[Propositions 18--20]{J-L}).  Since
$\sigma_p(M^*_z)$ has spherical symmetry, it suffices to check that
$\tilde{w}=(|w|,0, \cdots, 0) \in \mathbb B_{r(M_z)}$ belongs to
$\sigma_p(M^*_z)$. But
$\sum_{n \geq 0} {|\tilde{w}^n|^2}/{\|z^n\|^2_\beta}
= \sum_{n_1 \geq 0}
{|w|^{2n_1}}/{\|z^{n_1}_1\|^2_\beta}=
\sum_{n_1=0}^{\infty} {{m-1+n_1} \choose n_1} \, \bbeta^{-2}_{n_1} |w|^{2n_1} < \infty$,
and the claim follows.
This also shows that $\mathbb B_{r(M_z)} \subseteq \sigma_p(M^*_z).$
Finally, note that the maximal ball contained in the domain of
convergence of the above series is precisely $\mathbb B_{r(M_z)}.$

(3): This is clear from the proof of (2) and the spherical symmetry of $\sigma_p(M^*_z).$

(4):
First of all, it follows from \eqref{id} that
\[
m_\infty(T) = \lim_{k\to\infty} \inf_{\|h\|=1} \|T_\de^k h\|=i(M_z)
\]
(see \eqref{m-infty}). By Lemma~\ref{lem3.5n},
the open ball $\mathbb B_{i(M_z)}$ is disjoint from $\sigma_{ap}(T)$.
Since the approximate point-spectrum of $M_z$ is contained in the
Taylor spectrum, it follows that $\sigma_{ap}(T) \subset \overline{A}_{i(M_z), R(M_z)}$.

To prove the converse inclusion, consider the bounded linear operator $S_1:=M_{z_1}|_{\mathcal M},$
where $\mathcal M$ is the closure of $\mathbb C[z_1]$ in
$H^2(\beta).$
It is a one-variable weighted shift in the basis $\{z_1^j/\|z_1^j\|\}$, and
we can apply to it \cite[Theorem 1]{R}.
We get $\sigma_{ap}(S_1)=\{\alpha
\in \mathbb C : i(S_1) \leq |\alpha| \leq R(S_1)\},$ where $R(S_1),
i(S_1)$ are given by
\beqn
R(S_1)= \lim_{j \rar \infty} \sup_{k \geq
0} \Big( \frac{\bbeta_{k+j+1}}{\bbeta_{k+1}} \Big)^{\frac 1j}\,
\bigg( \frac{(k+2) \cdots (k+j+1)}{(k+m+1) \cdots
(k+j+m)}\bigg)^{\frac 1{2j}}
\eeqn
\beqn
i(S_1)=
\lim_{j \rar \infty}
\inf_{k \geq 0}
\Big(
\frac{\bbeta_{k+j+1}}{\bbeta_{k+1}}
\Big)^{\frac 1j}\,
\bigg(
\frac{(k+2) \cdots
(k+j+1)}{(k+m+1) \cdots (k+j+m)}\bigg)^{\frac 1{2j}}.
\eeqn
By Lemma \ref{lem4.3},
$i(S_1)=i(M_z)$ and $R(S_1)=R(M_z).$

The proof is divided into two cases:

$i(M_z) = R(M_z)$: If this happens then by the preceding discussion,
$i(S_1)=i(M_z)=R(M_z)=R(S_1).$ In particular, $\{w \in \mathbb C :
|w|=R(M_z)\}=\sigma_{ap}(S_1) \subseteq \sigma_{ap}(M_{z_1}).$ Now
by the projection property for the approximate point-spectrum
\cite[Pg. 18]{Cu}, the projection of $\sigma_{ap}(M_z)$ onto the
$z_1$-axis is precisely $\sigma_{ap}(M_{z_1}).$ Since $R(M_z) \in
\sigma_{ap}(M_{z_1}),$ there exists $w_2, \cdots, w_m \in \mathbb C$
such that $(R(M_z), w_2, \cdots, w_m) \in \sigma_{ap}(M_z).$ By (1)
above, $\sigma_{ap}(M_z) \subseteq \sigma(M_z) =
\overline{B}_{R(M_z)}.$ It follows that $w_2=\cdots = w_m=0,$ and
hence $(R(M_z), 0, \cdots, 0) \in \sigma_{ap}(M_z).$ Since
$\sigma_{ap}(M_z)$ has spherical symmetry, it contains the
degenerate ball-shell $\overline{A}_{i(M_z), R(M_z)}$.

$i(M_z) < R(M_z)$: Since the approximate point-spectrum is always
closed, it suffices to check that $A_{i(M_z), R(M_z)} \subseteq
\sigma_{ap}(M_z).$ Let $w \in A_{i(M_z), R(M_z)}.$ By the spherical
symmetry of $\sigma_{ap}(M_z),$ we may take $w$ of the form $(\|w\|,
0, \cdots, 0).$ We adapt the argument of \cite[Theorem 1]{R} to the
present situation. Choose numbers $a, b$ such that $i(M_z) < a <
\|w\| < b < R(M_z).$ Let $\epsilon
> 0$ be given. Choose positive integers $n, k$ such that $(\|w\|/b)^n <
\epsilon,$ $1/k < \epsilon$ and $\|z^{n+k+1}_1\|/\|z^{k+1}_1\| >
b^n$.  We further choose positive integers $p$ and $q$ such that
$(a/\|w\|)^p < \epsilon,$ $q > n+k$ and
$\|z^{p+q+1}_1\|/\|z^{q+1}_1\| < a^p.$ Now the argument in the
(\cite{R}, Proof of Theorem 1) actually yields
$\|(S_1-\|w\|I)f\| < \epsilon \|S_1\|\|f\|$, where $f \in \mathcal M$ is given by
\beqn
f(z_1)=\sum_{l=k+1}^{p + q+1}
\,
\|w\|^{k+1 - l} z^l_1 .
\eeqn
Set $g(z)=f(z_1),$ and note that $g \in H^2(\beta)$ and
$\|(M_{z_1}-\|w\|I)g\| < \epsilon \|M_{z_1}\|\|g\|.$ Further, since $1/k
< \epsilon,$ we have $\|M_{z_j}g\| < \epsilon \|M_{z_1}\|\|g\|$ for
$j=2, \cdots, m$. Since $\epsilon>0$ is arbitrary, $w\in \sigma_{ap}(T)$.

(5): The assertion about the essential spectrum is already obtained
in \cite[Theorem 4.5(2)]{G-H-X}. Alternatively, it may be deduced
from \cite[Lemma 4.7]{G-H-X}, Remark \ref{le4.5} and Lemma \ref{Pro4.1ess}.
\end{proof}

%

We remark that for an essentially commuting tuple $T$ satisfying
$\sigma_e(T)=\partial \mathbb B$, the $C^*$ algebra it generates
can be described using the results of \cite{G}.

Let $T$ be an essentially normal, spherical $m$-tuple.
It follows from Lemma \ref{Pro4.1ess} that the essential spectrum of $T$ is connected if and only if the essential
spectrum of $\sum_{j=1}^m T^*_j T_j$ is connected. If in addition,
$T$ is a multi-shift, then this always happens as seen above. In view
of this, it is interesting to note that there exists a jointly
hyponormal $2$-shift with disconnected essential spectrum
\cite[Theorem 2.5]{Cu-Yo}.

We close the section with the following question.

\begin{question}
Calculate the essential spectrum of any spherical multiplication $m$-tuple $M_z$.
Is it always connected?
\end{question}

As it is shown in \cite[Example 3.7.7]{LauNeu-book}, one always has
$\sigma_e(M_z)=\sigma_{ap}(M_z)=\bar A_{i(M_z), R(M_z)}$ if $M_z=M_{z_1}$
is a $1$-tuple.

\section{The Membership of Cross-commutators in the Schatten Classes}

\label{sec:4n}


In this section, we discuss the so-called $p$-essential normality of spherical tuples.
Recall that an $m$-tuple $T$ of commuting bounded linear operators $T_1, \cdots, T_m$ is {\it $p$-essentially normal} if the cross-commutators $[T^*_i, T_l]$ belong to the Schatten $p$-class for all $j, l=1, \cdots, m.$
As before, we put $\de_k=\bbeta_{k+1}/\bbeta_k$, $k\in \mathbb{N}$.
Throughout this section, we assume that $m \geq 2.$

\begin{remark}
\label{rem-compact}
An $m$-variable weighted shift $T : \{w^{(i)}_n\}$
is compact if and only if $\lim_{|n| \rar \infty}w^{(i)}_n = 0$ for
all indices $i=1, \dots, m$ \cite[Proposition 6]{J-L}. It follows that a spherical $m$-variable weighted shift
is compact if and only if $\lim_{k\rar \infty} \dt_k=0$.
\end{remark}



The main part of this Section is devoted to the proof of the following criterion of when
the cross-commutators of a spherical weighted shift
belong to the Schatten class $\mathcal S^p$.

\begin{theorem} \label{Thm_Sp}
Let $M_z$ be a bounded spherical multiplication
$m$-tuple in $H^2(\beta)$, so that the norm in $H^2(\beta)$ is given
by \eqref{H2-beta-L2B} for a certain sequence $\bbeta_0, \bbeta_1,
\bbeta_2, \cdots, $ of positive numbers. Let $1\le p \leq \infty$.
Then the following statements are equivalent:
\smallskip

\textbf{(1)} The self-commutators $[M^*_{z_j}, M_{z_j}]$
belong to the Schatten class $\mathcal S^p$ for all $j$, $1\le j\le m$;

\textbf{(2)} The cross-commutators $[M^*_{z_j}, M_{z_l}]$
belong to the Schatten class $\mathcal S^p$ for all indices $j,l$;

\textbf{(3)}
\vskip-.5cm
%
%
%
\bq
\label{expr-Sp}
\sum_{k=1}^{\infty}
\de_k^{2p}\,
k^{m-p-1}
+
\sum_{k=1}^{\infty} \big|
\de_k^2 - \de_{k-1}^2
 \big|^p \, k^{m-1}< \infty.
\eq
\end{theorem}

We refer to \cite{Kapt} for some related results, such as the membership
of $I-\sum M_{z_j}^*M_{z_j}$ in classes $S_p$, see Proposition 6.9 and other results
in Section 6 of the cited work.

The following notation will be used. We will say that two quantities $F_k$, $G_k$, depending on
$k\in \mathbb N$, are comparable, and write $F_k  \approx G_k$ $(k\to \infty)$
if there exist positive constants $A$, $B$, $k_0$ (that may depend on $m$ and $p$)
such that $A G_k \le  F_k \le B G_k$ for all $k\ge k_0$.

\begin{lemma} \label{Eq-Lp-2}
Let $1\le p<\infty$. Then for  $j\ne l$,
$$
\summ
{n \in \mathbb N^{m},}
{|n|=k,  \, n_j>0}
n_j^{p/2} n_l^{p/2}
\approx k^{p+m-1}
\enspace \text{as $k\to\infty$.}
$$
\end{lemma}

\begin{proof}
By symmetry, it suffices to consider the case $j=1$, $l=2$.
One has
$$
\summ
{n \in \mathbb N^{m},}
{|n|=k,  \, n_1>0}
n_1^{p/2} n_2^{p/2}
= \frac 1{m-1} \;
\sum_{r=2}^m \;
\summ
{n \in \mathbb N^{m},}
{|n|=k,  \, n_1>0}
n_1^{p/2} n_r^{p/2}
\approx
\summ
{n \in \mathbb N^{m},}
{|n|=k,  \, n_1>0}
n_1^{p/2}
\;
\Big(\sum_{r=2}^m  n_r\Big)^{p/2}.
$$
Hence
\begin{align*}
\summ
{n \in \mathbb N^{m},}
{|n|=k,  \, n_1>0}
n_1^{p/2} n_2^{p/2}
& \approx
\sum_{|n|=k,  \, n_1>0}
n_1^{p/2}(k-n_1)^{p/2}
=
\sum_{j=1}^{k}
{\textstyle  {k-j+m-2
\choose k-j}}\; j^{p/2} (k-j)^{p/2} \\
& \approx
\sum_{j=1}^{k}
j^{p/2} (k-j)^{m-2+(p/2)}
\approx k^{p+m-1} \int_0^1 x^{p/2}(1-x)^{m-2+(p/2)}\, dx,
\end{align*}
which gives the assertion of the Lemma.
\end{proof}

\begin{lemma} \label{lem-sum}
Let $1\le p <+\infty$. For any $j$, $1\le j\le m$, and any $s\in \mathbb R$, one has
\bq
\label{eq-sum}
\sum_{n\in \mathbb N^m,\, |n|=k}
\big|sn_j-1\big|^p \approx
k^{p+m-1}\, |s|^p + k^{m-1},
\eq
where the constants involved in the relation $\approx$ can depend on $p, m$ but not on $k$ and $s$.
\end{lemma}

\begin{proof}
Denoting $l=n_j$,
we get
$$
\sum_{n\in \mathbb N^m,\, |n|=k}
\big|sn_j-1\big|^p
=
\sum_{l=0}^k
\textstyle{ k-l+m-2 \choose m-2}
\displaystyle
|sl-1|^p
\approx
\sum_{l=0}^k
(k-l+1)^{m-2}|sl-1|^p.
$$
So the estimate in one direction is trivial:
$$
\sum_{n\in \mathbb N^m,\, |n|=k}
\big|sn_j-1\big|^p
\le
\sum_{l=0}^k
(k+1)^{m-2}|sl-1|^p
\le C
\sum_{l=0}^k
(k+1)^{m-2}(|sl|^p+1)
\le C_1
\big(k^{p+m-1}\, |s|^p + k^{m-1}\big).
$$
To prove the reverse estimate, define the interval $I(s)$ as
follows: $I(s)=[k/2, k]$ if $|s|k\ge 6$ and
$I(s)=[0, k/18]$ if $|s|k < 6$. One has
$$
|sl-1|\ge \frac{1+|s|l}2\enspace ~\text{for}\enspace l\in I(s).
$$
Indeed, in the first case
$|sl-1|\ge |sl|-1 \ge (1+|s|l)/2$ and in the second case,
$|sl-1|\ge 1-|sl| \ge (1+|s|l)/2$.
So for some positive constants $C_2$, $C_3$ one gets
\begin{multline*}
\sum_{n\in \mathbb N^m,\, |n|=k}
\big|sn_j-1\big|^p
\ge
\sum_{l\in \mathbb Z\cap I(s)}
(k-l+1)^{m-2}|sl-1|^p
\\
\displaystyle
\ge
C_2\,
\sum_{l\in \mathbb Z\cap I(s)}
(k-l+1)^{m-2}(1+|s|^pl^p)
\ge
C_3\,
\big(k^{p+m-1}\, |s|^p + k^{m-1}\big). \qquad \qedhere
\end{multline*}
\end{proof}

\begin{proof}[Proof of Theorem \ref{Thm_Sp}]
We will use the following formulas, which are easy to deduce
from \eqref{w-i-n}.

Let $1\le j\le m$.
Then
\bq
\label{comm_sph}
[T^*_j, T_j]\, e_n=
\begin{cases}
\Big[\frac {(n_j+1)\dt_{|n|}^2}{|n|+m}-
\frac {n_j\dt_{|n|-1}^2}{|n|+m-1}\Big]
e_n, \quad n_j >0, \\
\frac {\dt_{|n|}^2}{|n|+m}\; e_n
, \qquad \qquad \qquad \qquad n_j =0.
\end{cases}
\eq
If $1\le j , l\le m$ and $j\ne l$, then
\bq
\label{comm_sph2}
[T^*_j, T_l]\, e_n=
\begin{cases}
\sqrt{n_j(n_l+1)}\;
\bigg[
\frac {\dt_{|n|}^2}{|n|+m}-
\frac {\dt_{|n|-1}^2}{|n|+m-1}
\bigg] e_{n+\eps_l-\eps_j}, \quad n_j >0, \\
\;0, \qquad \qquad \qquad \quad \qquad \qquad \qquad \qquad \qquad \enspace  n_j =0.
\end{cases}
\eq
By symmetry, to prove that \textbf{(1)} is equivalent to \textbf{(3)},
it suffices to give two-sided estimates of the self-commutator
$[T_m^*, T_m]$.
Equation \eqref{comm_sph} gives
$$
\begin{aligned}
\big\|[T_m^*, T_m]\big\|_{\mathcal S^p}^p & =
\sum_{k=1}^\infty
\dm{
\summ
{n \in \mathbb N^{m},}
{|n|=k,  \, n_m>0}
\;
\bigg|
\frac
{(n_m+1) \dt_k^2} {k+m}
 -
\frac {n_m \dt_{k-1}^2} {k+m-1}
\bigg|^p }
+
\sum_{n \in \mathbb N^{m},\; n_m=0}\,
\frac
{\dt_{|n|}^{2p}}
{(|n|+m)^p}    \\
& \approx
\sum_{k=1}^\infty
\dm{
\summ
{n \in \mathbb N^{m},}
{|n|=k,  \, n_m>0}
\;
\bigg|
n_m\,
\Big(
\frac {\dt_{k}^2}{k+m} - \frac{\dt_{k-1}^2}{k+m-1}
\Big)+ \frac {\dt_{k}^2}{k+m}
\bigg|^p
}
+
\sum_{k=0}^\infty
(k+1)^{m-2-p} \,
\dt_k^{2p}\, .
\end{aligned}
$$
By substituting $s/t$ for $s$ in \eqref{eq-sum}, one gets
\[
\sum_{n\in \mathbb N^m,\, |n|=k}
\big|n_m s-t\big|^p \approx
k^{p+m-1}\, |s|^p + k^{m-1} |t|^p
\]
(where the constants involved in the relation $\approx$ do not depend on
$s, t\in \mathbb R$ and $k$).
It follows that
\[
\big\|[T_m^*, T_m]\big\|_{\mathcal S^p}^p \approx
\sum_{k=1}^\infty
\Big|
\frac {\dt_{k}^2}{k+m} - \frac{\dt_{k-1}^2}{k+m-1}
\Big|^p \,
k^{p+m-1}
+
\sum_{k=0}^\infty
\frac
{\dt_k^{2p}} {(k+m)^p} \, k^{m-1}\, .
\]
Now, by applying the two-sided estimate
\begin{multline}
\label{two-s}
C_1\bigg(
\frac {\big|\dt_{k}^2-\dt_{k-1}^2\big|^p}{(k+m)^p} \,
- \,\frac{\dt_{k-1}^{2p}}{(k+m)^p(k+m-1)^p}
\bigg)
\le
\Big|
\frac {\dt_{k}^2}{k+m} \,-\, \frac{\dt_{k-1}^2}{k+m-1}
\Big|^p  \\
\le
C_2\bigg(
\frac {\big|\dt_{k}^2-\dt_{k-1}^2\big|^p}{(k+m)^p}
\, + \,\frac{\dt_{k-1}^{2p}}{(k+m)^p(k+m-1)^p}\bigg)\;,
\end{multline}
where $C_1$ and $C_2$ are positive constants,
it is easy to see that
the relation $\big\|[T_m^*, T_m]\big\|_{\mathcal S^p}<\infty$
is equivalent to \textbf{(3)}.

It is obvious that \textbf{(2)} implies \textbf{(1)}, so it only remains to prove
that \textbf{(3)} implies that the commutators
$[T_j^*, T_l]$ belong to $\mathcal S_p$ whenever the indices $j$ and $l$
are distinct. This follows easily from \eqref{comm_sph2}, Lemma \ref{Eq-Lp-2}
and \eqref{two-s}.
\end{proof}

Next we derive several consequences of Theorem~\ref{Thm_Sp}, which
are motivated by the so-called
cut-off phenomenon in the Berger-Shaw
theory \cite[ Proposition 5.3]{Ar}, \cite[Proposition 3]{D-V},
\cite[Theorem 1.1]{F-X} (refer to \cite{Y} for a detailed account of
this phenomenon).

\begin{corollary}
\label{cor-comp-Sp}
Let $M_z$ be a bounded spherical multiplication $m$-tuple on $H^2(\be)$.
If all commutators $[M_{z_j}^*, M_{z_k}]$ belong to $S_p$, where
$1\le p <\infty$, then either the operators $M_{z_j}$ are compact or $p>m$.
\end{corollary}

\begin{proof}
To simplify notation, we put $\rr_k=\de_k^2$.
It suffices to show that if
all commutators $[M_{z_j}^*, M_{z_k}]$ are in $S_m$, then
$\rr_k\to 0$ as $k\to\infty$ (see Remark \ref{rem-compact}).
Suppose, to the contrary, that $[M_{z_j}^*, M_{z_k}]$ are in $S_m$,
but $\rr_k$ do not tend to zero. Then there exist an $\eps>0$ and a
sequence $k_1, k_2, \dots$ such that
$\rr_{k_j}>\eps$ for all $j\in \mathbb{N}$.
By Theorem~\ref{Thm_Sp},
\bq
\label{ineq_Sm}
\sum_k \rr_k^{m}\, k^{-1} + \sum_k | \rr_{k+1}-\rr_k |^m\, k^{m-1}<\infty,
\eq
and, as we will now show, it leads to a contradiction.
Choose $N$ so large that
\[
\sum_{l=k_N}^\infty | \rr_{l+1}-\rr_l |^m\, l^{m-1} < \Big(\frac \eps 2\Big)^m.
\]
Take any $j\ge N$ and any $k$ such that $k_j\le k\le 2 k_j$.
Since $\sum_{l=k_j}^{2k_j-1} \frac 1 l \le 1$,
we get
\begin{align*}
|\rr_k-\rr_{k_j}| & \le \sum_{l=k_j}^{k-1} | \rr_{l+1}-\rr_l | \\
 & \le
\Big(
\sum_{l=k_j}^{k-1} | \rr_{l+1}-\rr_l |^m l^{m-1}
\Big)^{\frac 1 m}
\cdot
\Big(
\sum_{l=k_j}^{k-1}  \frac 1 l
\Big)^{\frac{m-1} m }
 <
\frac \eps 2 \cdot 1 = \frac \eps 2\, .
\end{align*}
Hence $\rr_k> \eps/2$ for all $k$ in the range $k_j\le k\le 2 k_j$ for all $j\ge N$.
This implies that
\[
\sum_{k=k_j}^\infty  \rr_k^{m} \,k^{-1}   \ge
\sum_{k=k_j}^{2k_j}  \rr_k^{m} \,k^{-1}   \ge \Big(\frac \eps
2\Big)^m\, \sum_{k=k_j}^{2k_j}  k^{-1} \ge \frac{1}{2}\Big(\frac
\eps 2\Big)^m
\]
for all $j\ge N$. Therefore the first sum in \eqref{ineq_Sm} diverges.

This contradiction implies that, in fact, $\rr_k$ should tend to $0$.
\end{proof}

%
%

\begin{corollary}
\label{cor:de_k-1} Suppose that
 the sequence $\{\de_k\}$ does not tend to zero and
$|\de_{k+1}-\de_k|\le C/k$ for some constant $C$.
Then the
commutators $[M_{z_j}^*, M_{z_k}]$ belong to $S_p$ if and only if $p>m$.
\end{corollary}

\begin{proof}
Since $\{\de_k\}$ does not tend to zero. by Remark \ref{rem-compact}, none of $M_{z_1}, \cdots, M_{z_m}$ is compact.
If $[M_{z_j}^*, M_{z_l}]\in S_p$ for all $j, l$, then
by Corollary \ref{cor-comp-Sp}, $p>m$. The converse
statement follows immediately from Theorem \ref{Thm_Sp}.
\end{proof}

In particular, the statement of this Corollary holds if
$|\de_k-1|\le C/k$ for some constant $C$. It can also be applied to
sequences $\bbeta_k$ like $\bbeta_k=C_1\exp(C_2 k^\alpha)$, where
$C_1, C_2$ and $\al$ are constants.

We end this section with the following question.

\begin{question}
Give a characterization of all (strongly)
spherical $m$-tuples $T$ such that $\ker(D_{T^*})$ is finite-dimensional and is
cyclic for $T$, in terms of some free parameters (similarly to
Theorems ~\ref{th:w-sh} and~\ref{th:spher-tpl}). Can our results on the
calculation of parts of the spectrum and on membership
of cross-commutators $[T_j, T^*_j]$ in $S_p$ be generalized
to this subclass of spherical $m$-tuples?
\end{question}

\section{Special Classes of Spherical Tuples}


Recall that an $m$-tuple $S=(S_1, \cdots, S_m)$ of commuting
operators $S_i$ in ${\mathcal B}({\mathcal H})$ is {\it jointly subnormal}
if there exist a Hilbert space ${\mathcal K}$ containing ${\mathcal
H}$ and an $m$-tuple $N=(N_1, \cdots, N_m)$ of commuting normal
operators $N_i$ in ${\mathcal B}({\mathcal K})$ such that $N_ih =
S_ih$ for every $h \in {\mathcal H}$ and $1\leq i\leq m.$

An $m$-tuple $S=(S_1, \cdots, S_m)$ of commuting operators $S_i$ in
${\mathcal B}({\mathcal H})$ is {\it jointly hyponormal} if the $m
\times m$ matrix $([T^*_j, T_i])_{1 \leq i, j \leq m}$ is positive
definite, where $[A, B]$ stands for the commutator $AB-BA$ of $A$
and $B.$ It is not difficult to see that a jointly subnormal tuple is always
jointly hyponormal \cite{At-0}, \cite{Cu-1}.


\begin{definition}
\label{df5.1} Fix an integer $q \geq 1$ and put
\beq
\label{5.26}
B_q(Q_T)
\mathrel{\mathop:}=
\sum_{s=0}^q
(-1)^{s} {q \choose s} Q^{s}_T(I)
\eeq
(see \eqref{QT}, \eqref{QkT}).
We say that $T$ is a
{\it joint $q$-contraction} (respectively,
{\it joint $q$-expansion}) if
$B_q(Q_T) \ge 0$ (respectively,
$B_q(Q_T) \le 0$).
\end{definition}

We say that $T$ is a {\it joint $q$-hyperexpansion} if $T$ is a
joint $k$-expansion for all $k=1, \cdots, q.$ Also, $T$ is said
to be a {\it joint complete hyperexpansion} if $T$ is a
joint $q$-hyperexpansion for all $q \geq 1.$
If $B_q(Q_T)=0$, then $T$ is a {\it joint $q$-isometry}.
If $m=1$ then we drop the prefix $1$- and term joint in all the above definitions. \\



The Bergman $m$-shift is jointly subnormal while the
Drury-Arveson $m$-shift is a joint $m$-isometry \cite{G-R}.
The Szeg\"o $m$-shift being a joint isometry is jointly subnormal. It is also a joint $q$-isometry for any $q \ge 1.$

\begin{remark}
\label{rk5.2}
Let $T$ be a spherical $m$-tuple.
Assume further that $T$ is a joint $p$-isometry or a joint
$2$-hyperexpansion. Then the approximate point-spectrum
$\sigma_{ap}(T)$ of $T$ is a subset of the unit sphere \cite[Proposition 3.4]{Ch-Sh}.
Since $\sigma_{ap}(T)$ is always non-empty, by its spherical symmetry, it must
be the entire unit sphere.
\end{remark}

\begin{theorem}
\label{le5.4} Let $T : \{w^{(i)}_n\}$ be a
spherical
%
%
$m$-variable weighted shift.  Let $T_{\delta} : \{\delta_{k}\}_{ k
\in \mathbb N}$ be the  one-variable weighted shift associated with
$T$ (see the Definition \ref{df4.1}). Then we have the following
statements:
\begin{enumerate}
\item $T$ is jointly subnormal if and only if
$T_{\delta}$ is subnormal.
\item $T$ is a joint $q$-isometry if and only if
$T_{\delta}$ is a $q$-isometry.
\item $T$ is a joint $q$-expansion if and only if
$T_{\delta}$ is a $q$-expansion.
\item $T$ is a joint complete hyperexpansion if and only if
$T_{\delta}$ is a complete hyperexpansion.
\item $T$ is jointly
hyponormal if and only if $T_{\delta}$ is hyponormal.
\end{enumerate}
\end{theorem}
\begin{proof}
The desired conclusions in (2)-(4) follow immediately from (\ref{id}).

To see (1), without loss of generality, we may assume that
$Q_T(I)=T^*_1T_1 + \cdots + T^*_mT_m \leq I.$ By \cite[Theorem
5.2]{At-2}, an operator $m$-tuple $T$ such that $Q_T(I) \leq I$ is
jointly subnormal if and only if
$$\sum_{j=0}^p (-1)^j {p \choose j}
\sum_{|\alpha|=j}\frac{j!}{\alpha!}\|T^{\alpha}f\|^2 \geq 0$$ for
every $f \in \mathcal H$ and every $p, k \in \mathbb N$. Now (1) may
be derived from \eqref{id}.


To see (5), let us recall first that $T_\de$ is hyponormal if and
only if and only if $\{\delta_k\}_{k \in \mathbb N}$ is an
increasing sequence. Suppose first that $T_\de$ is
hyponormal.
By \cite[Theorem 6.1]{Cu-1},
$T$ is jointly hyponormal if and only if the matrix
\beqn
P=\left(
w^{(i)}_{n + \epsilon_j}w^{(j)}_{n + \epsilon_i}-w^{(i)}_n w^{(j)}_n
\right)_{1 \leq i, j \leq m}
\eeqn
%
%
is positive definite for every $n \in \mathbb N^m.$
By \eqref{w-i-n},
$$
w^{(i)}_n=\delta_{|n|} \alpha^{(i)}_n~~(n
\in \mathbb N^m, i=1, \cdots, m),
$$
where
$\alpha^{(i)}_n:=\sqrt{\frac{n_i + 1}{|n|+m}}$ is the weight
multi-sequence of the Szeg\"o $m$-shift.
It is easy to see that the matrix
\beqn
Q=\left(\alpha^{(i)}_{n +
\epsilon_j}\alpha^{(j)}_{n +
\epsilon_i}-\alpha^{(i)}_n \alpha^{(j)}_n \right)_{1 \leq i, j
\leq m}
\eeqn
is positive definite for every $n \in \mathbb N^m.$
Let $P_{ij}, Q_{ij}$ denote the $(i, j)$th entry of $m \times m$
matrices $P, Q$ respectively.
It follows that
\beqn
P_{ij}= \delta^2_{|n|+1}Q_{ij} +
(\delta^2_{|n|+1}-\delta^2_{|n|})\alpha^{(i)}_n\alpha^{(j)}_n
\eeqn
Since $\{\delta_k\}_{k \in \mathbb N}$ is an
increasing sequence, $P$ is positive definite.

Conversely, suppose $T$ is jointly hyponormal. Then it follows from
\cite[Lemma 4.10]{Ch-Sh} that $Q^2_T(I) \geq Q_T(I)^2,$ where
$Q_T(X)=T^*_1XT_1 + \cdots + T^*_mXT_m~(X \in B(\mathcal H)).$
It is immediate from \eqref{Qken} that $\{\delta_k\}_{k \in \mathbb N}$ is an
increasing sequence.
\end{proof}

Let $p>0$, and let $M_{z, p}$ be as introduced in Example \ref{ex1.3}.
Note that the sequence $\delta_k$ there is given by $\delta^2_k=(k+m)/(k+p)$. It is now easy to deduce from Theorem \ref{le5.4}(5) that the tuple
$M_{z, p}$ is jointly hyponormal if and only if $p \ge m$.
As we already mentioned there, for $p \ge m,$ $M_{z, p}$ is actually jointly subnormal
(see \cite[Theorem 9.8]{Kapt} for a closely related fact).
Here are a few more consequences of Theorem \ref{le5.4}:
\begin{enumerate}

\item
$M_{z, p}$ is a spherical joint $2$-expansion if and only if $m-1 \le p \le m$.

\item $M_{z, p}$ is a spherical joint $q$-isometry if and only if $p$ is a
positive integer, $p\le m$ and $q\ge m-p+1$ (see Proposition~\ref{pr5.5} below).


\end{enumerate}

It is a trivial consequence of Corollary \ref{cor-comp-Sp} and Theorem~\ref{le5.4}(5) that
for a spherical hyponormal $m$-variable weighted shift $T$, the cross-commutators
can belong to $S_p$ only if $p>m$.
We give an example of a spherical, jointly hyponormal $2$-variable
weighted shift for which none of the self-commutators belongs to
the Schatten class $\mathcal S^p$ for any $p<\infty$.
\begin{example}
Define inductively the sequence $\{\rho_k\}_{k \in \mathbb N}$ as follows:
\beqn
\rho_0=1, \rho_{k+1}=\rho_k + \eta_k~(k \in \mathbb N),
\eeqn
where $\eta_k = 1/2^l$ if $k$ is of the form $2^{2^l};$ and $0$ otherwise.
Obviously, $\{\rho_k\}_{k \in \mathbb N}$ is increasing.
We next define $\bbeta_k$ inductively by setting $\bbeta_0=1$ and $\bbeta_{k+1}=\bbeta_k \sqrt{\rho_k}~(k \in \mathbb N).$
Then the $2$-variable weighted shift $T=(T_1,T_2)$ with weight multi-sequence
\beqn
w^{(i)}_n = \frac{\bbeta_{|n|+1}}{\bbeta_{|n|}} \sqrt{\frac{n_i+1}{|n|+2}}~(n \in \mathbb N^2, i=1, 2)
\eeqn
is bounded, spherical and jointly hyponormal.
Note that
\beqn
\sum_{k=1}^{\infty} k
\Big|\frac{\bbeta^2_{k+1}}{\bbeta^2_{k}}-\frac{\bbeta^2_{k}}
{\bbeta^2_{k-1}}
 \Big|^p = \sum_{k=1}^{\infty} k  \eta^p_k \geq 2^{2^l}\frac{1}{2^{lp}} \rar \infty \eeqn
as $l \rar \infty.$ By Theorem \ref{Thm_Sp}, $[T^*_j, T_j]$ does not
belong to the Schatten class $\mathcal S^p$ for any $p <\infty$ and any $j =1, 2$.
\end{example}

In a similar way, one give an example of a non-compact spherical
$2$-variable weighted shift $T$ such that the
corresponding one-variable shift $T_\de$ is a contraction
(or, equivalently, the sequence $\{\bbeta_k\}$ decays), but the cross-commutators do not belong to $S_p$ for any $p$.

For a sequence $\{f_k\}_{k=0}^\infty$, we put $\nabla f_k=f_{k+1}-f_k$.
In what follows, we denote
\[
\ga_k=\bbeta_k^2.
\]





The following is certainly known (see, for instance, \cite[pg 50]{Sh-At}).
We include a
short proof for reader's convenience.
\begin{proposition}
\label{pr5.5}
Let $\{\bbeta_k\}_{k=0}^\infty$ be a $1$-variable sequence and
$M_z$ the multiplication operator by $z$ acting on the
($1$-variable) space $H^2(\bbeta)$.

\begin{enumerate}
\item
$M_z$ is a $q$-isometry
if and only if there is a polynomial $S$ of degree $q-1$ or less such that
$\ga_k=S(k)$ for all $k\in \mathbb N$.

\item $M_z$ is
a $q$-expansion if and only if $(-1)^q \nabla^q\ga_k \le 0$ for any $k \in \mathbb N.$

\end{enumerate}

\end{proposition}

\begin{proof}
(1): By the definition, $M_z$ is a $q$-isometry if and only if
$
\sum_{s=0}^q
(-1)^{s} {q \choose s} M_z^{^*s} M_z ^s =0
$.
The left hand part is a diagonal operator in the basis
$\{z^k\}$ (for any weights $\{\bbeta_k\}$). Hence
$M_z$ is a $q$-isometry iff
$
\sum_{s=0}^q
(-1)^{s} {q \choose s} \inp{M_z^s z^k}{M_z^s z^k} =0
$
for any $k\in\mathbb N$, which happens iff $\nabla^q \ga_k\equiv 0$.
This gives (1).

(2): This can be proved along the lines of the verification of (1),
and hence we skip it.
\end{proof}

Now it follows from Theorem \ref{le5.4}(2)
that a spherical $m$-tuple $M_z$ is a $q$-isometry
if and only if the corresponding scalar sequence
$\{\bbeta_k\}$ satisfies $\bbeta_k^2=S(k)$, $k\in \mathbb N$,
for some polynomial $S$, whose degree is less or equal than $q-1$ .
In particular, we get the following fact.
\begin{corollary}
Let $M_z$ be a spherical $q$-isometry. Then $[M_{z_j}^*, M_{z_l}]$ is in $S_p$ for all $j, l$
if and only if $p>m$.
\end{corollary}

\begin{proof}
Assuming that the $m$-tuple $M_z$ is a $q$-isometry, we get
a polynomial $S$ of degree $d\le q-1$ such that
$\be_k^2=S(k)$, $k\in \mathbb N$.
Then $\de_k^2-1\sim d/k$, and we can apply Corollary~\ref{cor:de_k-1} to get our
statement.
\end{proof}

\begin{proposition}
\label{pr-conc}
Let $T$ be a $m$-variable spherical weighted
shift and $\{\bbeta_k\}$, $\{\de_k\}$ be the corresponding
$1$-variable sequences. Suppose that $\de_k\le C$ and that the sequence $\{\de_k\}$ does
not tend to zero. Put $\ga_k=\bbeta_k^2$ (as before).

\begin{enumerate}

\item If $\nabla^2 \ga_k\le 0$, then
all commutators $[T_j^*, T_k]$ are in $S_p$ iff $p>m$.

\item If $\nabla^3 \ga_k \le 0$ and $\bbeta_k>C>0$, then it is also true that
all commutators $[T_j^*, T_k]$ belong to $S_p$ iff $p>m$.

\end{enumerate}
\end{proposition}

It will follow from the proof that it suffices to assume that
the inequality $\nabla^2 \ga_k\le 0$ or
$\nabla^3 \ga_k \le 0$ holds except for a finite number of indices.

\begin{proof}[Proof of Proposition~\ref{pr-conc}]
First observe that
\bq
\label{nabla2}
\de_{k+1}^2-\de_{k}^2
=
\frac {\ga_{k+2}}{\ga_{k+1}} - \frac {\ga_{k+1}}{\ga_{k}}
=
\frac {\nabla^2\ga_k}{\ga_k}
-
\frac {\nabla\ga_{k+1}\nabla\ga_k}{\ga_{k+1}\ga_k}\, .
\eq

\textbf{Proof of (1):} Suppose that $\nabla^2\ga_k\le 0$.
It follows from Richter's Lemma \cite[Lemma 6.9]{H-K-Z} that
$\nabla\ga_{k}\ge 0$ for all $k$. Since $\{\ga_k\}$ is a concave
sequence,
\[
\nabla\ga_k = \ga_{k+1} - \ga_k  \le \frac {\ga_k- \ga_{0}} k \le \frac {\ga_k} k
\]
for all $k\ge 1$. Hence
\[
|\nabla^2\ga_k| = - \nabla^2\ga_k
= \nabla \ga_k- \nabla \ga_{k+1} \le
\nabla \ga_k \le \frac {\ga_k} k\, .
\]
Therefore, by \eqref{nabla2},
\[
|\de_{k+1}^2 - \de_k^2|
\le
\frac {|\nabla^2\ga_k|}{\ga_k}
+
\frac {|\nabla\ga_{k+1}||\nabla\ga_k|}{\ga_{k+1}\ga_k}
\le
\frac 1 k + \frac 1 k \cdot \frac 1 {k+1}
\le
\frac 2 k\,  , \qquad k\ge 2.
\]
So the assertion (1) follows from Corollary~\ref{cor:de_k-1}.

\smallskip
\textbf{Proof of (2):}
We are assuming that $\nabla^3\ga_k\le 0$ and that
$\ga_k\ge C>0$. Then there is an index $k_0$ such that the sign of
$\nabla^2\ga_k$ is constant for $k\ge k_0$. If
$\nabla^2\ga_k\le 0$ for all $k\ge k_0$, the proof is as above.
So we can suppose that there is $k_0\in \mathbb N$ such that $0<\nabla^2\ga_k\le \nabla^2\ga_{k_0}$
for $k\ge k_0$. Hence $\{\nabla\ga_k\}_{k\ge k_0}$ is a growing sequence.
Let us distinguish two opposite cases.

\textbf{Case A:} {\it $\nabla\ga_k>0$ for large indices $k>k_0$.}
Then $\nabla\ga_k > C_1 >0$, and therefore $\ga_k \ge C_2 k$ for
large indices $k$, where  $C_2>0$. Next let $k> 2k_0$. If $k$ is
even, then
\[
\ga_k > \ga_k- \ga_{k/2}=\sum_{\ell=k/2}^{k-1} \nabla \ga_\ell \ge \frac k 2 \nabla \ga_{k/2}
\]
Since $\{\nabla \ga_\ell\}$ is concave,
\[
\frac {\nabla \ga_k + \nabla \ga_0} 2 \le \nabla \ga_{k/2} < \frac {2 \ga_k}k\, .
\]
Therefore
\[
\frac {\nabla \ga_k } {\ga_k} \le
\frac 4 k +\frac {|\nabla \ga_0| } {\ga_k} \le \frac {C_3} k\, .
\]
In the same way, one gets that $\frac {\nabla \ga_k } {\ga_k} \le \frac {C_3} k$ for odd $k$, $k>2k_0$
(just replace indices $0, k/2$ by $1, (k+1)/2$ in the above estimates).
Since $\{\nabla^2\ga_k\}$ is bounded,
\eqref{nabla2} implies the estimate $|\de_{k+1}-\de_k|\le C/k$, and we are done.

\textbf{Case B:}
{\it $\nabla\ga_k\le 0$ for all $k \geq k_0.$}
%
%
Hence $\{|\nabla\ga_k|\}$ decays for $k\ge k_0$, and
\[
(k-k_0)|\nabla\ga_k| \le
\sum_{\ell=k_0}^{k-1} |\nabla\ga_\ell|
=\ga_{k_0}-\ga_{k}\le \ga_{k_0}
\]
for $k > k_0$, so that $|\nabla \ga_k|\le C/k$. Next, for $k > k_0$,
$\nabla^2\ga_k = |\nabla\ga_k| - |\nabla\ga_{k+1}| \le |\nabla\ga_k| \le C/k$, and once again, we obtain that
$|\de_{k+1}-\de_k|\le C/k$ by using \eqref{nabla2} and the assumption
$\bbeta_k\ge \operatorname{const} >0$.

Therefore in both cases A and B, Corollary~\ref{cor:de_k-1} implies
assertion (2).
\end{proof}

\begin{example}
In the part 2) of the last Proposition, one cannot drop the assumption $\bbeta_k\ge C>0$.
Indeed, define $\{\bbeta_k\}$ by $\bbeta_{2k}^2={12}^{-k}$ and
$\bbeta_{2k+1}^2={12}^{-k}/3$, $k\ge 0$. Then $(\nabla^3 \ga)_{2k} =-(2/9)\cdot{12}^{-k}$
and $(\nabla^3 \ga)_{2k+1} =-(23/144)\cdot{12}^{-k}$, $k\ge 0$, so that
$\nabla^3 \ga_k < 0$ for all $k$. On the other side,
$\de_{k+1}^2-\de_k^2 =(-1)^{k+1}/12$ for any $k$, so that
in this case, the tuple $M_z$ is bounded, but is not essentially normal
(and therefore the self-commutators do not belong to any $S_p$).
\end{example}

\begin{remark}  Let $T : \{{w^{(i)}_n}\}_{ n \in \mathbb
N^m}$ be a spherical  $m$-variable weighted shift and let $T_{\delta}: \{\delta_k\}_{k \in \mathbb N}$ be the shift associated with $T.$
Suppose that $\{\delta_k\}_{k \in \mathbb N}$ converges to a non-negative number $\lambda.$ Then by Remark \ref{le4.5}, $T$ is essentially normal.
Moreover, by Theorem \ref{th:3.1}(5), the essential spectrum of $T$ is $\partial \mathbb B_{\lambda}.$
This happens whenever $T$ is jointly hyponormal, a joint $q$-isometry or a joint $2$-hyperexpansion.
In case $T$ is jointly hyponormal, $\lambda=\|T_{\delta}\|$ while in the remaining two cases
$\lambda$ is equal to $1.$

\end{remark}


\end{document}